\documentclass{amsart}

\usepackage[utf8]{inputenc} 
\usepackage[T1]{fontenc} 
\usepackage[english]{babel} 
\usepackage{mathrsfs}

\usepackage{pgfplots} 
\pgfplotsset{compat=newest}

\theoremstyle{plain} 
\newtheorem{theorem}{Theorem} 
\newtheorem{lemma}[theorem]{Lemma} 
\newtheorem{corollary}[theorem]{Corollary} 

\theoremstyle{definition} 
\newtheorem*{definition}{Definition}

\theoremstyle{remark} 
\newtheorem*{remark}{Remark} 
\newtheorem*{question}{Question}

\DeclareMathOperator{\mre}{Re}

\begin{document} 

\title[Norms of composition operators on the $H^2$ space of Dirichlet series]{Norms of composition operators on the \\ $H^2$ space of Dirichlet series} 
\date{\today} 

\author{Ole Fredrik Brevig} 
\address{Department of Mathematical Sciences, Norwegian University of Science and Technology (NTNU), NO-7491 Trondheim, Norway} 
\email{ole.brevig@math.ntnu.no} 
\author{Karl-Mikael Perfekt} 
\address{Department of Mathematics and Statistics, University of Reading, Reading RG6 6AX, United Kingdom} 
\email{k.perfekt@reading.ac.uk} 
\begin{abstract}
	We consider composition operators $\mathscr{C}_\varphi$ on the Hardy space of Dirichlet series $\mathscr{H}^2$, generated by Dirichlet series symbols $\varphi$. We prove two different subordination principles for such operators. One concerns affine symbols only, and is based on an arithmetical condition on the coefficients of $\varphi$. The other concerns general symbols, and is based on a geometrical condition on the boundary values of $\varphi$. Both principles are strict, in the sense that they characterize the composition operators of maximal norm generated by symbols having given mapping properties. In particular, we generalize a result of J.~H.~Shapiro on the norm of composition operators on the classical Hardy space of the unit disc. Based on our techniques, we also improve the recently established upper and lower norm bounds in the special case that $\varphi(s) = c + r2^{-s}$. A number of other examples are given.
\end{abstract}

\maketitle

\section{Introduction} \label{sec:intro} 
In a seminal paper, Gordon and Hedenmalm \cite{GH99} obtained a characterization of the bounded composition operators on the Hardy space of Dirichlet series $\mathscr{H}^2$. A Dirichlet series $f(s)=\sum_{n\geq1} a_n n^{-s}$ belongs to $\mathscr{H}^2$ if
\[\|f\|_{\mathscr{H}^2}^2 := \sum_{n=1}^\infty |a_n|^2<\infty.\]
The present paper is devoted to investigating the norms of composition operators on $\mathscr{H}^2$, in relation to certain mapping properties of the generating symbols.

By the Cauchy--Schwarz inequality, note that $\mathscr{H}^2$ is a space of analytic functions in the half-plane $\mathbb{C}_{1/2}$, where
\[\mathbb{C}_\theta := \{s \in \mathbb{C}\,:\, \mre{s}>\theta\}.\]
Hence, if $\varphi$ is an analytic function mapping $\mathbb{C}_{1/2}$ into itself, then $\mathscr{C}_\varphi f := f\circ \varphi$ defines an analytic function in $\mathbb{C}_{1/2}$, for every $f$ in $\mathscr{H}^2$.

However, the symbol $\varphi$ has to satisfy additional arithmetical conditions to ensure that $f \circ \varphi$ is a Dirichlet series, and further mapping properties are required for $f\circ\varphi$ to have square summable coefficients for every $f$ in $\mathscr{H}^2$. The main result\footnote{The statement on uniform convergence is from \cite[Thm.~3.1]{QS15}.} of \cite{GH99} shows that $\mathscr{C}_\varphi$ defines a bounded composition operator on $\mathscr{H}^2$ if and only if $\varphi$ belongs to the Gordon--Hedenmalm class $\mathscr{G}$. 
\begin{definition}
	The Gordon--Hedenmalm class, denoted $\mathscr{G}$, consists of the functions $\varphi\colon \mathbb{C}_{1/2}\to\mathbb{C}_{1/2}$ of the form
	\[\varphi(s) = c_0 s + \sum_{n=1}^\infty c_n n^{-s} =: c_0s + \varphi_0(s),\]
	where $c_0$ is a non-negative integer and the Dirichlet series $\varphi_0$ converges uniformly in $\mathbb{C}_\varepsilon$ for every $\varepsilon>0$, satisfying the following mapping properties: 
	\begin{itemize}
		\item[(a)] If $c_0=0$, then $\varphi_0(\mathbb{C}_0)\subseteq \mathbb{C}_{1/2}$. 
		\item[(b)] If $c_0\geq1$, then either $\varphi_0\equiv 0$ or $\varphi_0(\mathbb{C}_0)\subseteq \mathbb{C}_0$. 
	\end{itemize}
\end{definition}
That $\varphi$ is defined on $\mathbb{C}_0$ in the characterization is initially surprising, and proving the necessity of this is perhaps the most difficult aspect of \cite{GH99}. That this requirement is not unreasonable can be understood in view of Carlson's formula 
\begin{equation}\label{eq:carlson} 
	\sum_{n=1}^\infty |a_n|^2 n^{-2\sigma} = \lim_{T\to\infty} \frac{1}{2T}\int_{-T}^T |f(\sigma+it)|^2\,dt, 
\end{equation}
which is valid when $f$ converges uniformly for $\mre{s}\geq\sigma$. In particular, if the Dirichlet series $f$ converges uniformly for $\mre{s}\geq0$, then we may choose $\sigma = 0$ to express the $\mathscr{H}^2$-norm as an $L^2$-average of the boundary values $f(it)$ of $f$. In general, elements $f \in \mathscr{H}^2$ only converge in $\mathbb{C}_{1/2}$, and there are certainly no boundary values with respect to $\mathbb{C}_{0}$. However, there is a very useful notion of generalized boundary values via vertical limit functions, discussed in Section~\ref{sec:vertical}.

Recall from \cite{GH99} that composition operators $\mathscr{C}_\varphi$ generated by symbols with $c_0\geq1$ always satisfy that $\|\mathscr{C}_\varphi\|=1$. Since we are interested in non-trivial norm estimates, we shall exclusively consider the case $c_0=0$, when the symbol $\varphi$ is a Dirichlet series. The mapping properties of $\varphi$ that we will refer to consist of the point $\omega = \varphi(+\infty)$ and the domain $\Omega = \varphi(\mathbb{C}_0^\ast)$, where $\mathbb{C}_0^\ast := \mathbb{C}_0 \cup \{+\infty\}$. 

We will prove two different kinds of subordination principles. If $\varphi$ and $\psi$ have the same mapping properties, we will say that $\mathscr{C}_\varphi$ is subordinate to $\mathscr{C}_\psi$ whenever it holds that
\[\|\mathscr{C}_\varphi f\|_{\mathscr{H}^2}\leq \|\mathscr{C}_\psi f\|_{\mathscr{H}^2}\]
for every $f \in \mathscr{H}^2$.

In the first part of the paper, we will consider affine symbols. These are symbols of the form 
\begin{equation}\label{eq:affinesymbol} 
	\varphi_\mathbf{c}(s) = c + \sum_{j=1}^d c_j p_j^{-s}, 
\end{equation}
where $\mathbf{c} = (c_1,\ldots, c_d)$ and $(p_j)_{j\geq1}$ denotes the increasing sequence of prime numbers. In this case, the compactness of $\mathscr{C}_{\varphi_{\mathbf{c}}}$ has previously been studied in \cite{Bayart03,FQV04,QS15}. Note that $\varphi_{\mathbf{c}}$ is in $\mathscr{G}$ if and only if $\mre{c}>1/2$ and $\mre{c}-1/2 \geq \sum_{j\geq1} |c_j| =: r$. For an affine symbol, we see that $c = \varphi_{\mathbf{c}}(+\infty)$ and Kronecker's theorem implies that $\varphi_{\mathbf{c}}(\mathbb{C}_0^\ast) = \mathbb{D}(c,r)$, where
\[\mathbb{D}(c,r) := \left\{s \in \mathbb{C}\,: |s-c|<r\right\},\]
see Lemma~\ref{lem:affinemap}.

In terms of the mapping properties and the norm of $\mathscr{C}_{\varphi_{\mathbf{c}}}f$, we can without loss of generality assume that $c_j\geq0$ for $1 \leq j \leq d$. Suppose that $\mathbf{b} = (b_1, \ldots, b_d)$ is another vector with non-negative elements and $\sum_{j\geq 1} b_j = r$. Then, if $\mathbf{c}$ majorizes $\mathbf{b}$, $\mathbf{b} \prec \mathbf{c}$, we will prove in Theorem~\ref{thm:affinesubord} that $\mathscr{C}_{\varphi_{\mathbf{b}}}$ is subordinate to $\mathscr{C}_{\varphi_{\mathbf{c}}}$, and moreover that the following are equivalent: 
\begin{enumerate}
	\item[(a)] $\mathbf{b}$ is a permutation of $\mathbf{c}$. 
	\item[(b)] $\|\mathscr{C}_{\varphi_\mathbf{b}} f\|_{\mathscr{H}^2} = \|\mathscr{C}_{\varphi_\mathbf{c}} f\|_{\mathscr{H}^2}$ for every $f \in \mathscr{H}^2$. 
	\item[(c)] $\|\mathscr{C}_{\varphi_\mathbf{b}}\| = \|\mathscr{C}_{\varphi_\mathbf{c}}\|$. 
\end{enumerate}

In particular, the symbols $\varphi(s) = c + r p_j^{-s}$ generate composition operators of strictly maximal norm in the class of affine symbols with the same mapping properties. Muthukumar, Ponnusamy and Queff\'elec \cite{MPQ18} have recently investigated the norm of these operators. It is of course sufficient to only consider the case $\varphi(s) = c + r 2^{-s}$. They established the estimates 
\begin{equation}\label{eq:mpqbounds} 
	\zeta(2\mre{c}) \leq \|\mathscr{C}_\varphi\|^2 \leq \zeta(1+\xi) 
\end{equation}
where $\xi := (\mre{c}-1/2) + \sqrt{(\mre{c}-1/2)^2-r^2}$. The lower bound in \eqref{eq:mpqbounds} is actually a general lower bound which holds for any Dirichlet series $\varphi \in \mathscr{G}$, 
\begin{equation}\label{eq:genlower} 
	\|\mathscr{C}_\varphi\|^2 \geq \zeta(2\mre{c_1}), 
\end{equation}
see \cite{GH99}.

The full statement of Theorem~\ref{thm:affinesubord} is a bit more precise. As a corollary, we will see that the the lower bound \eqref{eq:genlower} is the best possible, even when considering only affine symbols with $\varphi(\mathbb{C}_0^\ast)=\mathbb{D}(c,r)$. On the other hand, we shall prove in Theorem~\ref{thm:pointeval} that the lower bound is attained if and only if $\varphi(s) \equiv c_1$. 

In certain cases, we will also improve the estimates \eqref{eq:mpqbounds} for $\varphi(s)=c+r2^{-s}$. In Theorem~\ref{thm:adjointlower} we obtain the new lower bound $\|\mathscr{C}_\varphi\|^2 \geq \xi^{-1}$. This constitutes a major improvement when $\mre{c}-1/2$ is small; the difference between the upper and lower estimates is now bounded (by $1$), whereas it was previously unbounded. In Theorem~\ref{thm:newupper}, we will combine our techniques with a result from \cite{Brevig17} to improve the upper estimate in \eqref{eq:mpqbounds}, showing that
\[\|\mathscr{C}_\varphi\|^2 \leq \frac{\zeta(1+2\xi)+\zeta(1+\xi)}{2}\]
when $\mre{c}-1/2=r\geq \alpha_0$, for a specific value $\alpha_0\approx 1.5$.

In the second half of the paper, we turn our attention to general Dirichlet series symbols $\varphi \in \mathscr{G}$. Consider a simply connected domain $\Omega \subseteq \mathbb{C}_{1/2}$ with Jordan curve boundary on the Riemann sphere, fix $\omega \in \Omega$, and let $\psi(s)=\Theta(2^{-s})$, where $\Theta$ is a Riemann map from $\mathbb{D}$ to $\Omega$ with $\Theta(0)=\omega$. By standard methods, cf. \cite{Brevig17, GH99}, it is fairly easy to establish that if $\varphi(\mathbb{C}_0^\ast) \subseteq \Omega$ and $\varphi(+\infty)=\omega$, then $\mathscr{C}_\varphi$ is subordinate to $\mathscr{C}_\psi$.

In analogy with Theorem~\ref{thm:affinesubord}, we will determine which symbols $\varphi$ with the prescribed mapping properties satisfy that $\|\mathscr{C}_\varphi\|_{\mathscr{H}^2} = \|\mathscr{C}_\psi\|_{\mathscr{H}^2}$. For the classical Hardy space of the unit disc, the analogous problem has been solved by J.~H.~Shapiro \cite{Shapiro00}. He showed that the norm equality holds if and only if the symbol generating the composition operator is an inner function, see Theorem~\ref{thm:shapiro2} for a precise statement.

We call a Dirichlet series $f \in \mathscr{H}^2$ inner if its generalized boundary value $f^\ast(\chi)$ is unimodular for almost every $\chi \in \mathbb{T}^\infty$. We refer to Section~\ref{sec:vertical} for an explanation. Our analogue of Shapiro's theorem, Theorem~\ref{thm:inner}, is the following: if $\varphi$, $\psi$ and $\Theta$ are as above, then the following are equivalent. 
\begin{enumerate}
	\item[(a)] $\Theta^{-1} \circ \varphi$ is inner. 
	\item[(b)] $\|\mathscr{C}_\varphi f\|_{\mathscr{H}^2} = \|\mathscr{C}_{\psi} f\|_{\mathscr{H}^2}$ for every $f \in \mathscr{H}^2$. 
	\item[(c)] $\|\mathscr{C}_\varphi\| = \|\mathscr{C}_{\psi}\|$. 
\end{enumerate}

Of course, the most difficult implication is (c) $\implies$ (a). To prove it, we first consider the classical setting, improving a key estimate from \cite{Shapiro00} by showing that it can be made uniform in the ``non-innerness'' of the symbol. By a trick, we are then able to apply this improved uniform estimate ``on average'', thereby extending it to composition operators on $\mathscr{H}^2$. Combining this estimate with the earlier mentioned Theorem~\ref{thm:pointeval} yields the desired implication.

\subsection*{Organization} 
\begin{itemize}
	\item In Section~\ref{sec:vertical} we compile some preliminary results regarding vertical limit functions and non-tangential boundary values for $\mathscr{H}^2$, having the study of composition operators in mind. 
	\item Section~\ref{sec:affine} is devoted to composition operators generated by affine symbols. We prove Theorem~\ref{thm:affinesubord} and revisit the upper bound in \eqref{eq:mpqbounds}. 
	\item In Section~\ref{sec:rpk} we use partial reproducing kernels estimate to investigate lower bounds for the norms of composition operators on $\mathscr{H}^2$. In addition to proving Theorem~\ref{thm:pointeval} and Theorem~\ref{thm:adjointlower}, we discuss a question from \cite{MPQ18} on whether the norm of the composition operator generated by $\varphi(s)=c+r2^{-s}$ can be computed by testing $\mathscr{C}_\varphi$ or its adjoint operator on reproducing kernels. 
	\item In Section~\ref{sec:compdisc} we consider composition operators and inner functions for the Hardy space of the unit disc. Our goal is to obtain two versions of a key estimate from \cite{Shapiro00}. The first is used in the proof of Theorem~\ref{thm:inner}, while the second plays a role in the proof of Theorem~\ref{thm:newupper}. 
	\item Section~\ref{sec:inner} contains the proofs of Theorem~\ref{thm:newupper} and Theorem~\ref{thm:inner}. 
	\item In Section~\ref{sec:examples} we present three examples related to the results of Section~\ref{sec:affine} and Section~\ref{sec:inner}. 
\end{itemize}

\subsection*{Acknowledgments} We are grateful to Horatio Boedihardjo, Titus Hilberdink, and Herv\'e Queff\'elec for helpful discussions.

\section{Vertical limit functions} \label{sec:vertical} 
The purpose of this preliminary section is to extract some useful information about vertical limit functions and composition operators from \cite{Bayart02,GH99,HLS97,QS15}. 

Let us begin by emphasizing that we cannot use Carlson's formula \eqref{eq:carlson} in general. To obtain norm estimates it is, of course, sufficient to consider $\mathscr{C}_\varphi f$ for Dirichlet polynomials $f$ only, since they are dense in $\mathscr{H}^2$. In this case, $\mathscr{C}_\varphi f$ is a bounded analytic function in $\mathbb{C}_0$; in particular, $\mathscr{C}_\varphi f$ is uniformly convergent in $\mathbb{C}_\varepsilon$ for every $\varepsilon>0$ and has non-tangential boundary values almost everywhere on the imaginary axis. However, Saksman and Seip \cite{SS09} have shown that even under these assumptions, we cannot in general recover the $\mathscr{H}^2$-norm as the $L^2$-average of the non-tangential boundary values. Therefore, for a general symbol $\varphi$, we do not expect to obtain a complete understanding of the norm of its composition operator solely from the non-tangential boundary values.

To introduce the vertical limit functions, we let $\mathbb{T}^\infty$ denote the countably infinite Cartesian product of the torus $\mathbb{T}:=\{z\in\mathbb{C}\,:\,|z|=1\}$. The infinite torus $\mathbb{T}^\infty$ forms a compact commutative group under coordinate-wise multiplication. Its Haar measure $m_\infty$ is the countably infinite product measure generated by the normalized Lebesgue arc measure of $\mathbb{T}$, denoted $m$. For
\[\chi = (\chi_1,\chi_2,\chi_3,\ldots) \in \mathbb{T}^\infty\]
we define the character $n \mapsto \chi(n)$ to be completely multiplicative in $n$, setting $\chi(p_j) = \chi_j$. For $f(s)=\sum_{n\geq1} a_n n^{-s}$ and $\chi\in\mathbb{T}^\infty$, the vertical limit function $f_\chi$ is defined by
\[f_\chi(s) := \sum_{n=1}^\infty a_n \chi(n) n^{-s}.\]
Note that the vertical translation $T_\tau f(s):= f(s+i\tau)$, $\tau \in \mathbb{R}$, corresponds to $\chi(n)=n^{-i\tau}$. The name vertical limit function is justified by \cite[Lem.~2.4]{HLS97}, which asserts that the functions $f_\chi$ are precisely those obtained from the Dirichlet series $f$ by taking a limit of vertical translations, 
\begin{equation}\label{eq:vtrans} 
	f_\chi(s) = \lim_{k\to\infty} T_{\tau_k}f(s). 
\end{equation}
The convergence in \eqref{eq:vtrans} is uniform on compact subsets of the half-plane where $f$ converges uniformly. The proof of this fact relies on Kronecker's theorem, which analytically encodes the ``arithmetical independence'' of the prime numbers.

The vertical limit functions $f_\chi$ sometimes have better properties than the original function $f$. As explained in \cite[Sec.~4.2]{HLS97}, if $f$ is in $\mathscr{H}^2$, the Dirichlet series $f_\chi$ converges in $\mathbb{C}_0$ for almost every $\chi \in \mathbb{T}^\infty$, and the non-tangential boundary value 
\begin{equation}\label{eq:chiboundary} 
	f^\ast(\chi) := \lim_{\sigma \to 0^+} f_\chi(\sigma) 
\end{equation}
exists for almost every $\chi \in \mathbb{T}^\infty$. Moreover, $f^\ast$ is in $L^2(\mathbb{T}^\infty)$ and satisfies 
\begin{equation}\label{eq:chinorm} 
	\|f\|_{\mathscr{H}^2}^2 = \int_{\mathbb{T}^\infty} |f^\ast(\chi)|^2 \,dm_\infty(\chi). 
\end{equation}
Hence \eqref{eq:chiboundary} explicitly provides the Bohr correspondence, which is an isometric isomorphism between $\mathscr{H}^2$ and the Hardy space of the infinite torus $H^2(\mathbb{T}^\infty)$. Extending $\chi$ in a completely multiplicative fashion to act on the positive rationals $\mathbb{Q}_+$, any $F$ in $ L^2(\mathbb{T}^\infty)$ has a Fourier series $F(\chi) = \sum_{q \in \mathbb{Q}_+} a_q \chi(q)$, and
\[\|F\|^2_{L^2(\mathbb{T}^\infty)} = \sum_{q \in \mathbb{Q}_+} |a_q|^2.\]
The Hardy space $H^2(\mathbb{T}^\infty)$ is the subspace of $L^2(\mathbb{T}^\infty)$ of functions $F$ such that $a_q = 0$ whenever $q \in \mathbb{Q}_+ \backslash \mathbb{N}$.

We will now discuss the connection between composition operators and vertical limit functions \cite{GH99}. Recall that the Dirichlet series $\varphi$ is in $\mathscr{G}$ if it converges uniformly in $\mathbb{C}_\varepsilon$ for every $\varepsilon>0$ and $\varphi(\mathbb{C}_0^\ast)\subseteq\mathbb{C}_{1/2}$. If $f$ is in $\mathscr{H}^2$, this implies that $\mathscr{C}_\varphi f$ converges uniformly in $\mathbb{C}_\varepsilon$ for every $\varepsilon>0$ and that
\[(\mathscr{C}_\varphi f)_\chi(s) = (f\circ \varphi)_\chi(s) = f \circ \varphi_\chi(s) = \mathscr{C}_{\varphi_\chi} f(s).\]
This implies that $\|\mathscr{C}_{\varphi_\chi} f \|_{\mathscr{H}^2}=\|\mathscr{C}_\varphi f\|_{\mathscr{H}^2}$, and thus $\varphi_\chi$ is in $\mathscr{G}$ for every $\chi \in \mathbb{T}^\infty$. Moreover, the image of the extended half-plane $\mathbb{C}_0^\ast$ is invariant under vertical limits. As far as we know, this claim, certainly known to experts, has not been explicitly stated in the literature. 
\begin{lemma}\label{lem:compchi} 
	Suppose that $\varphi$ is in $\mathscr{G}$ and fix $\chi \in \mathbb{T}^\infty$. Then $\varphi_\chi(\mathbb{C}_0^\ast)=\varphi(\mathbb{C}_0^\ast)$. 
\end{lemma}
\begin{proof}
	Since $\chi(1) = 1$, we have that
	\[\varphi(+\infty) = c_1 = \varphi_\chi(+\infty).\]
	If $\varphi$ is identically constant we are done. Suppose therefore that $\varphi$ is not identically constant. Fix $w \in \mathbb{C}_0$ and let $K$ be closed disk in $\mathbb{C}_0$ which contains $w$ and satisfies that
	\[M = \inf_{s \in \partial K} |\varphi_\chi(s)-\varphi_\chi(w)|>0.\]
	Since $\varphi$ converges uniformly in $\mathbb{C}_\varepsilon$ for every $\varepsilon>0$, we get from \eqref{eq:vtrans} that there is a sequence of real numbers $\tau_k$ such that $\varphi(s+i \tau_k) \to \varphi_\chi(s)$ uniformly for $s \in K$. Hence there is some $\tau_k$ such that
	\[\frac{M}{2}\geq |\varphi(s+i\tau_k)- \varphi_\chi(s)| = |\varphi(s+i\tau_k)-\varphi_\chi(w)- (\varphi_\chi(s)-\varphi_\chi(w))|\]
	for every $s \in K$. By Rouch\'{e}'s theorem we conclude that there is $s_k \in K + i\tau_k$ such that $\varphi(s_k)=\varphi_\chi(w)$. 
\end{proof}

Let us now recall from \cite{Bayart02} how to obtain vertical limit functions for symbols of composition operators. We cannot appeal directly to the discussion above, since there are Dirichlet series in $\mathscr{G}$ which are not in $\mathscr{H}^2$. The Cayley transform 
\begin{equation}\label{eq:cayley} 
	\mathscr{T}(z) := \frac{1-z}{1+z} 
\end{equation}
is a conformal map from $\mathbb{D}$ onto the half-plane $\mathbb{C}_0$. Note that if $\varphi$ is a Dirichlet series in $\mathscr{G}$, then
\[\Phi(s)=\mathscr{T}^{-1}(\varphi(s)-1/2)\]
is a Dirichlet series which converges uniformly in $\mathbb{C}_\varepsilon$ for every $\varepsilon>0$ and $|\Phi(s)|< 1$ in $\mathbb{C}_0$. Hence $\Phi$ is in $\mathscr{H}^\infty$, the space of all Dirichlet series that converge to bounded analytic functions in $\mathbb{C}_0$. The norm is given by
\[\|\Phi\|_{\mathscr{H}^\infty} := \sup_{\mre{s}>0}|\Phi(s)|.\]
We recall from \cite{HLS97} that $\mathscr{H}^\infty$ coincides with the multiplier algebra of $\mathscr{H}^2$. In particular, $\mathscr{H}^\infty \subseteq \mathscr{H}^2$. Hence $\Phi$ has a non-tangential boundary value \eqref{eq:chiboundary} for almost every $\chi \in \mathbb{T}^\infty$. Since $\mathscr{T}$ extends to a homeomorphism on the Riemann sphere $\mathbb{C}^\ast$, we conclude that $\varphi$ has non-tangential boundary values
\[\varphi^\ast(\chi) := \lim_{\sigma \to 0^+} \varphi_{\chi}(\sigma),\]
for almost every $\chi \in \mathbb{T}^\infty$.

We conclude the present section with an extension of \cite[Lem.~4.1]{QS15}, removing the assumption that $\varphi(\mathbb{C}_0^\ast)$ is a bounded subset of $\mathbb{C}_{1/2}$. It shows that all information about the norm of $\mathscr{C}_\varphi$ is encoded in $\varphi^\ast$. 

For its statement, recall from \cite[Thm.~4.11]{HLS97} that if $f$ is in $\mathscr{H}^2$, then $f$ has non-tangential boundary values $f(1/2+it)$ almost everywhere on $\partial \mathbb{C}_{1/2}$. Furthermore, there is a universal constant $C \geq 1$ such that 
\begin{equation}\label{eq:locemb} 
	\int_0^1 |f(1/2 + it)|^2 \, dt \leq C\|f\|_{\mathscr{H}^2}^2. 
\end{equation}
Inequality \eqref{eq:locemb} furnishes an example of a Carleson measure $\mu$ for $\mathscr{H}^2$, that is, a measure $\mu$ on $\overline{\mathbb{C}_{1/2}}$ such that the inclusion $\mathscr{H}^2 \hookrightarrow L^2(\mu)$ is bounded. 
\begin{lemma}\label{lem:carlesoncomp} 
	If $\varphi$ is a Dirichlet series in $\mathscr{G}$ and $f$ is in $\mathscr{H}^2$, then 
	\begin{equation}\label{eq:carlesoncomp} 
		\|\mathscr{C}_\varphi f\|_{\mathscr{H}^2}^2 = \int_{\mathbb{T}^\infty} |f\circ \varphi^\ast(\chi)|^2\,dm_\infty(\chi). 
	\end{equation}
\end{lemma}
\begin{proof}
	We assume first that $f$ is a Dirichlet polynomial. Since $\varphi$ is in $\mathscr{G}$ we know that $f\circ\varphi$ is in $\mathscr{H}^2$ and hence that the boundary value $(f\circ \varphi)^\ast(\chi)$ exists for almost every $\chi \in \mathbb{T}^\infty$. Inserting this into \eqref{eq:chinorm}, we find that
	\[\|\mathscr{C}_\varphi f\|_{\mathscr{H}^2}^2 = \|f \circ \varphi\|_{\mathscr{H}^2}^2 = \int_{\mathbb{T}^\infty} |(f\circ\varphi)^\ast(\chi)|^2\,dm_\infty(\chi).\]
	Since $f$ is a Dirichlet polynomial and the non-tangential boundary value $\varphi^\ast(\chi)$ exists for almost every $\chi$, we conclude that $(f\circ\varphi)^\ast(\chi)=f\circ\varphi^\ast(\chi)$ holds for almost every $\chi$. Hence we have established \eqref{eq:carlesoncomp} when $f$ is a Dirichlet polynomial. 
	
	We now let $\mu_{\varphi^\ast}$ denote the push-forward of $m_\infty$ by $\varphi^\ast$, which for polynomials $f$ yields that
	\[\|\mathscr{C}_\varphi f\|_{\mathscr{H}^2}^2 = \int_{\mathbb{T}^\infty} |f\circ\varphi^\ast(\chi)|^2\,dm_\infty(\chi) = \int_{\overline{\mathbb{C}_{1/2}}}|f(s)|^2\,d\mu_{\varphi^\ast}(s).\]
	Since $\mathscr{C}_\varphi \colon \mathscr{H}^2\to\mathscr{H}^2$ is bounded, we find that $\mu_{\varphi^\ast}$ is a Carleson measure for $\mathscr{H}^2$. Additionally, since the reproducing kernel of $\mathscr{H}^2$ at the point $w\in\mathbb{C}_{1/2}$ is
	\[K_w(s) = \zeta(s+\overline{w}) = \frac{1}{s+\overline{w}-1}+O(1),\]
	a simple argument (see e.g.~\cite[Thm.~3]{OS12}) shows that $\mu_{\varphi^\ast}$ satisfies Carleson's condition for $H^2(\mathbb{C}_{1/2})$. In particular, $\nu = \mu_{\varphi^\ast}|_{\partial \mathbb{C}_{1/2}}$ is absolutely continuous with respect to the one-dimensional Lebesgue measure $m$, and the density $\frac{d\nu}{dm}$ is bounded. 
	
	Suppose that $(f_k)_{k\geq1}$ is a sequence of polynomials such that $f_k \to f \in \mathscr{H}^2$. Since $\mu_{\varphi^\ast}$ is a Carleson measure for $\mathscr{H}^2$, the sequence $(f_k)_{k\geq1}$ also has a limit in $L^2(\nu)$. But by the above observation and \eqref{eq:locemb}, the limit must coincide with the non-tangential boundary values of $f$ on the support of $\nu$. Hence, by the boundedness of $\mathscr{C}_\varphi$ we conclude that
	\[\|\mathscr{C}_\varphi f\|_{\mathscr{H}^2}^2 = \int_{\overline{\mathbb{C}_{1/2}}}|f(s)|^2\,d\mu_{\varphi^\ast}(s), \qquad f \in \mathscr{H}^2.\]
	This is equivalent to \eqref{eq:carlesoncomp}, by the definition of a push-forward measure. 
\end{proof}
As is well known, the proof of Lemma~\ref{lem:carlesoncomp} shows that questions about composition operators $\mathscr{C}_\varphi$ can be recast in terms of embedding problems. For example, in the special case that $\varphi(s) = c+r2^{-s}$, studied in \cite{MPQ18}, the upper estimate of \eqref{eq:mpqbounds} can be restated as
\[\int_{\overline{\mathbb{C}_{1/2}}}|f(s)|^2\,d\mu_{\varphi^\ast}(s) = \int_{\mathbb{T}} |f(c+r\chi_1)|^2 \, dm(\chi_1) \leq \zeta(1+\xi) \|f\|_{\mathscr{H}^2}^2, \qquad f \in \mathscr{H}^2.\]
We also mention \cite{Brevig17}, where the norms of composition operators were computed exactly through the associated Carleson embeddings, for a small family of operators. In the latter example, boundedness of the induced Carleson embeddings is easily seen to be seen to be equivalent to the embedding property \eqref{eq:locemb}, although the norms are different. In general, the Carleson measures of $\mathscr{H}^2$ arising from composition operators \cite{QS15} are much better understood than general Carleson measures \cite{PP18}. In the non-Hilbertian case of $\mathscr{H}^p$, $p\neq2$, defined in the next section, the situation is even more complicated \cite{BBHOP19, Harper17, OS12}.

\section{Composition operators generated by affine symbols} \label{sec:affine} 
Let $\varphi(s)=c + \sum_{j\geq1} c_j p^{-s}_j$ be an affine symbol of the form \eqref{eq:affinesymbol}. The terminology here is justified by the fact that $\varphi^\ast(\chi) = c + \sum_{j\geq1} c_j \chi_j$. We begin by computing the image of the extended half-plane $\mathbb{C}_0^\ast$ under $\varphi$. 
\begin{lemma}\label{lem:affinemap} 
	Let $\varphi$ be an affine symbol of the form \eqref{eq:affinesymbol} belonging to $\mathscr{G}$. Then $\varphi(\mathbb{C}_0^\ast ) = \mathbb{D}(c,r)$, where 
	\begin{equation}\label{eq:sumr} 
		r = \sum_{j=1}^\infty |c_j| \leq \mre{c}-1/2. 
	\end{equation}
\end{lemma}
\begin{proof}
	By Lemma~\ref{lem:compchi}, we can replace $\varphi$ with $\varphi_\chi$ without affecting $\varphi(\mathbb{C}_0^\ast)$. We begin by choosing $\chi \in \mathbb{T}^\infty$ such that $\chi(p_j) c_j \leq 0$ for every $j$, which is possible since $\chi(p_j)=\chi_j$. Since $\varphi_\chi(\mathbb{C}_0) \subseteq \mathbb{C}_{1/2}$ we find that
	\[\mre{c}- \sum_{j=1}^\infty |c_j|\, p_j^{-\sigma}>1/2\]
	for every $\sigma>0$. We let $\sigma\to0^+$ to see that the coefficient sequence is summable and satisfies \eqref{eq:sumr}. Furthermore, it is clear that $\varphi(\mathbb{C}_0^\ast)\subseteq \mathbb{D}(c,r)$ since $\varphi(+\infty)=c$ and
	\[|\varphi(s)-c|\leq \sum_{j=1}^\infty |c_j| p_j^{-\sigma}<r, \qquad s \in \mathbb{C}_0.\]
	Replacing $\chi = (\chi_1, \chi_2,\ldots)$ with $e^{i\theta} \chi = (e^{i\theta}\chi_1, e^{i\theta}\chi_2,\ldots)$, we observe that the set $\varphi(\mathbb{C}_0^\ast)-c$ is invariant under rotations. The conclusion now follows from the fact that
	\[\sigma \mapsto \sum_{j=1}^\infty |c_j|\,p_j^{-\sigma}\]
	maps $(0,\infty)$ onto $(0,r)$. 
\end{proof}

We will need a preliminary lemma before we proceed to the main result of this section. Since the lemma might be of independent interest, we state it for general Hardy spaces of Dirichlet series. Following \cite{Bayart02}, the Hardy space $\mathscr{H}^q$, $1\leq q < \infty$, is defined as the closure of Dirichlet polynomials $f(s) = \sum_{n=1}^N a_n n^{-s}$ in the Besicovitch norm
\[\|f\|_{\mathscr{H}^q} ^q := \lim_{T\to\infty} \frac{1}{2T} \int_{-T}^T |f(it)|^q \,dt.\]
We will rely on the facts that the $\mathscr{H}^q$-norm satisfies the triangle inequality, that it is invariant under permutations of the prime numbers, and that it is strictly convex for $q > 1$. The easiest way to establish these properties is to identify $\mathscr{H}^q$ with $H^q(\mathbb{T}^\infty)$ \cite[Thm.~2]{Bayart02}, as was described for $q=2$ in Section~\ref{sec:vertical}.

For given $d\geq1$ and $r>0$, let $\mathscr{L}(d,r)$ denote the family of sequences $\mathbf{c} = (c_1,c_2,\ldots,c_d)$ satisfying $c_j\geq0$ and $c_1+c_2+\cdots+c_d=r$. For every $\mathbf{c} \in \mathscr{L}(d,r)$, we consider the corresponding linear function
\[L_{\mathbf{c}}(s) := \sum_{j=1}^d c_j p_j^{-s}.\]
Let $\mathbf{c}^\downarrow$ denote the decreasing rearrangement of $\mathbf{c}$. We write $\mathbf{b} \prec \mathbf{c}$ if $\mathbf{c}$ majorizes $\mathbf{b}$, that is, if
\[\sum_{j=1}^k b_j^\downarrow \leq \sum_{j=1}^k c_j^\downarrow\]
for $k=1,2,\ldots,d-1$, with equality for $j=d$. We note that 
\begin{equation}\label{eq:seqdom} 
	\left(\frac{r}{d},\frac{r}{d},\ldots,\frac{r}{d}\right) \prec \mathbf{c} \prec (r,0,0,\ldots, 0) 
\end{equation}
for every $\mathbf{c} \in \mathscr{L}(d,r)$. 
\begin{lemma}\label{lem:convex} 
	Let $1 \leq q \leq \infty$. If $\mathbf{b}, \mathbf{c} \in \mathscr{L}(d,r)$ and $\mathbf{b} \prec \mathbf{c}$, then $\|L_{\mathbf{b}}\|_{\mathscr{H}^q}\leq \|L_{\mathbf{c}}\|_{\mathscr{H}^q}$. The inequality is strict if $\mathbf{b}$ is not a permutation of $\mathbf{c}$ and $1 < q < \infty$. 
\end{lemma}
\begin{proof}
	By the Birkhoff--von Neumann theorem, $\mathbf{b} \prec \mathbf{c}$ holds if and only if there is a finite number of permutations $(P_k)$ and non-negative weights $(\lambda_k)$ such that $\lambda_1 + \cdots + \lambda_K =1$ and
	\[\mathbf{b} = \sum_{k=1}^K \lambda_k P_k \mathbf{c}.\]
	By the triangle inequality and invariance under permutations of prime numbers, we obtain that
	\[\|L_\mathbf{b}\|_{\mathscr{H}^q} = \left\|\sum_{k=1}^K \lambda_k L_{P_k \mathbf{c}} \right\|_{\mathscr{H}^q} \leq \sum_{k=1}^K \lambda_k \|L_{P_k \mathbf{c}}\|_{\mathscr{H}^q} = \sum_{k=1}^K \lambda_k \|L_{\mathbf{c}}\|_{\mathscr{H}^q} = \|L_\mathbf{c}\|_{\mathscr{H}^q},\]
	which is the required inequality. If $1 < q < \infty$ and $\mathbf{b}$ is not a permutation of $\mathbf{c}$, the inequality is strict, owing to the strict convexity of $\mathscr{H}^q$. 
\end{proof}

The main result of this section consists of a partial subordination principle for the family of affine symbols that map the extended right half-plane onto the same disc, and a sharpened inequality for comparison with the maximal element of the family.

Recall from Section~\ref{sec:vertical} that we may replace $\varphi$ with $\varphi_\chi$ for any $\chi\in\mathbb{T}^\infty$ without changing its mapping properties or the norm of $\|\mathscr{C}_\varphi f\|_{\mathscr{H}^2}$, for any $f \in \mathscr{H}^2$. Hence we may assume that $c_j\geq0$. 
\begin{theorem}\label{thm:affinesubord} 
	Fix $c$ and $r$ such that $\mre{c}-1/2\geq r>0$ and let $d$ be a positive integer. For $\mathbf{c} \in \mathscr{L}(d,r)$, let
	\[\varphi_{\mathbf{c}}(s) := c + \sum_{j=1}^d c_j p_j^{-s}.\]
	Suppose that $\mathbf{b}, \mathbf{c} \in \mathscr{L}(d,r)$ and $\mathbf{b} \prec \mathbf{c}$. Then
	\[\|\mathscr{C}_{\varphi_\mathbf{b}} f\|_{\mathscr{H}^2}\leq \|\mathscr{C}_{\varphi_\mathbf{c}} f\|_{\mathscr{H}^2}, \qquad f \in \mathscr{H}^2.\]
	Furthermore, if $\mathbf{b}\prec\mathbf{c}$, the following are equivalent. 
	\begin{enumerate}
		\item[(a)] $\mathbf{b}$ is a permutation of $\mathbf{c}$. 
		\item[(b)] $\|\mathscr{C}_{\varphi_\mathbf{b}} f\|_{\mathscr{H}^2} = \|\mathscr{C}_{\varphi_\mathbf{c}} f\|_{\mathscr{H}^2}$ for every $f \in \mathscr{H}^2$. 
		\item[(c)] $\|\mathscr{C}_{\varphi_\mathbf{b}}\| = \|\mathscr{C}_{\varphi_\mathbf{c}}\|$. 
	\end{enumerate}
	Additionally, for every $\mathbf{c} \in \mathscr{L}(d,r)$ it holds that 
	\begin{equation}\label{eq:effectivesubord} 
		\|\mathscr{C}_{\varphi_\mathbf{c}} f\|_{\mathscr{H}^2}^2 \leq (1-C)|f(c)|^2 + C \|\mathscr{C}_\varphi f\|_{\mathscr{H}^2}^2, \qquad f \in \mathscr{H}^2, 
	\end{equation}
	where $\varphi(s) = c + r2^{-s}$ and $C = \|\mathbf{c}\|_{\ell^2}^2 / \|\mathbf{c}\|_{\ell_1}^2$. The estimate \eqref{eq:effectivesubord} also holds in the case when $d=\infty$. 
\end{theorem}
\begin{proof}
	Let $f$ be any function in $\mathscr{H}^2$. Following the notation of Lemma~\ref{lem:convex}, we write $\varphi_\mathbf{c}(s) = c + L_\mathbf{c}(s)$. We begin by Taylor expanding $f$ at $s=c$ to obtain
	\[\mathscr{C}_{\varphi_\mathscr{c}} f(s) = f(c + L_\mathbf{c}(s)) = \sum_{k=0}^\infty \frac{f^{(k)}(c)}{k!} \big(L_\mathbf{c}(s)\big)^k.\]
	The sequence $(L_\mathbf{c}^k)_{k=0}^\infty$ is orthogonal in $\mathscr{H}^2$, yielding that 
	\begin{equation}\label{eq:ncomp} 
		\|\mathscr{C}_{\varphi_\mathbf{c}} f\|_{\mathscr{H}^2}^2 = \sum_{k=0}^\infty \frac{|f^{(k)}(c)|^2}{(k!)^2} \|L_\mathbf{c}^k\|_{\mathscr{H}^2}^2 = \sum_{k=0}^\infty \frac{|f^{(k)}(c)|^2}{(k!)^2} \|L_\mathbf{c}\|_{\mathscr{H}^{2k}}^{2k}. 
	\end{equation}
	Hence if $\mathbf{b} \prec \mathbf{c}$, then it follows directly from Lemma~\ref{lem:convex} that $\|\mathscr{C}_{\varphi_\mathbf{b}}f \|_{\mathscr{H}^2} \leq \|\mathscr{C}_{\varphi_\mathbf{c}}f\|_{\mathscr{H}^2}$. It is also clear that (a) $\implies$ (b) $\implies$ (c). 
	
	To prove that $(c) \implies (a)$, suppose that $\mathbf{b}$ is not a permutation of $\mathbf{c}$. If $f$ is non-constant, so that there is some $k \geq 1$ for which $f^{(k)}(c) \neq 0$, then Lemma~\ref{lem:convex} actually shows that $\|\mathscr{C}_{\varphi_\mathbf{b}}f \|_{\mathscr{H}^2} < \|\mathscr{C}_{\varphi_\mathbf{c}}f\|_{\mathscr{H}^2}$. Note that the vector $\mathbf{b}$ must have two or more non-zero elements. Therefore $\mathscr{C}_{\varphi_{\mathbf{b}}}$ is a compact operator, by the results of \cite{FQV04}. In particular, $\mathscr{C}_{\varphi_{\mathbf{b}}}$ is norm-attaining. The general lower norm-bound \eqref{eq:genlower} shows that the norm is not attained at a constant function $f$. We conclude that $\|\mathscr{C}_{\varphi_{\mathbf{b}}}\| < \|\mathscr{C}_{\varphi_{\mathbf{c}}}\|$.
	
	It remains to prove \eqref{eq:effectivesubord}. Consider the final sum in \eqref{eq:ncomp} for $k\geq1$. In this case, since either $d<\infty$ or we have summable coefficients, we have that
	\[\|L_\mathbf{c}\|_{\mathscr{H}^{2k}}^{2k} = r^{2k} \left\|\frac{1}{r}\sum_{j=1}^d c_j p_j^{-s}\right\|_{\mathscr{H}^{2k}}^{2k}= r^{2k}\lim_{T\to\infty} \frac{1}{2T} \int_{-T}^T \left|\frac{1}{r}\sum_{j=1}^d c_j p_j^{-it}\right|^{2k}dt.\]
	Since $|r^{-1}L_\mathbf{c}(it)|\leq1$, the integral on the right hand side is non-increasing in $k$. This implies that
	\[\|L_\mathbf{c}\|_{\mathscr{H}^{2k}}^{2k} \leq r^{2k} \frac{\|L_\mathbf{c}\|_{\mathscr{H}^2}^2}{r^2} = \frac{\|\mathbf{c}\|_{\ell^2}^2}{\|\mathbf{c}\|_{\ell^1}^2} r^{2k}.\]
	The proof is completed by noting that if $\varphi(s)=c+r2^{-s}$, then $\|\varphi-c\|_{\mathscr{H}^{2k}}^{2k}=r^{2k}$. 
\end{proof}

We now present a simple proof of \cite[Lem.~3.7]{MPQ18} which also yields a new lower bound that will find use in the next section. It is inspired by an even simpler proof for the case $k=1$, shown to us by Horatio Boedihardjo. 
\begin{lemma}\label{lem:dkzeta} 
	Let $k$ be a non-negative integer. For $\sigma>1$ it holds that
	\[\frac{k!}{(\sigma-1)^k} \left(\zeta(\sigma)-1\right)\leq (-1)^k \frac{d^k}{d\sigma^k} \zeta(\sigma) \leq \frac{k!}{(\sigma-1)^k} \zeta(\sigma).\]
\end{lemma}
\begin{proof}
	The case $k=0$ is obvious. For $k\geq1$ we introduce an integral representation for $(\log{n})^k$ and change the order of summation and integration, to obtain that
	\[\sum_{n=1}^\infty \frac{(\log{n})^k}{n^\sigma} = \sum_{n=1}^\infty \frac{k}{n^\sigma} \int_1^n \frac{(\log{x})^{k-1}}{x}\,dx = k \int_1^\infty \frac{(\log{x})^{k-1}}{x^\sigma}\, F_\sigma(x) \,dx,\]
	where $F_\sigma(x) := x^{\sigma-1} \sum_{n\geq x} n^{-\sigma}$. Computing
	\[k \int_1^\infty \frac{(\log{x})^{k-1}}{x^\sigma}\,dx = \frac{k}{(\sigma-1)^k} \int_0^\infty x^{k-1} e^{-x}\,dx = \frac{k!}{(\sigma-1)^k},\]
	it is sufficient to prove that $\zeta(\sigma)-1 \leq F_\sigma(x) \leq \zeta(\sigma)$ holds for $x\geq1$. Clearly, $x \mapsto F_\sigma(x)$ is increasing on the interval $(m,m+1)$ for every positive integer $m$. Hence we obtain upper and lower bounds of $F_\sigma(x)$ by considering, respectively,
	\begin{align*}
		U_\sigma(m) &:= \lim_{x\to m^-} F_\sigma(x) = m^{\sigma-1} \sum_{n=m}^\infty \frac{1}{n^\sigma} = \sum_{j=1}^\infty \sum_{n=0}^{m-1}\frac{1}{m} \left(j+\frac{n}{m}\right)^{-\sigma},
		\intertext{and}
		L_\sigma(m) &:= \lim_{x\to m^+} F_\sigma(x) = m^{\sigma-1} \sum_{n=m+1}^\infty \frac{1}{n^\sigma} = \sum_{j=1}^\infty \sum_{n=1}^{m}\frac{1}{m} \left(j+\frac{n}{m}\right)^{-\sigma}.
	\end{align*}
	For each $j$, we recognize the inner summands as the left and right Riemann sums with a uniform partition of length $m^{-1}$ for the integral
	\[\int_j^{j+1} y^{-\sigma}\,dy.\]
	Since $y \mapsto y^{-\sigma}$ is decreasing on the interval $(j,j+1)$, a simple geometric argument yields that $U_\sigma(m) \leq U_\sigma(1) = \zeta(\sigma)$ and $L_\sigma(m) \geq L_\sigma(1)= \zeta(\sigma)-1$. 
\end{proof}
\begin{remark}
	Since $y \mapsto y^{-\sigma}$ is convex on $(1,\infty)$, it actually holds that the sequences $(U_\sigma(m))_{m\geq1}$ and $(L_\sigma(m))_{m\geq1}$ are decreasing and increasing, respectively. This is a stronger statement than we require in the proof of Lemma~\ref{lem:dkzeta}. Monotonicity results for Riemann sums of convex and concave functions have probably been rediscovered many times (see e.g.~\cite{BJ00}). 
\end{remark}

The next result is \cite[Thm.~3.8]{MPQ18}. We present a different, but ultimately equivalent, proof that follows our approach to Theorem~\ref{thm:affinesubord}. For certain choices of the parameters, we will improve this estimate in Theorem~\ref{thm:newupper}. 
\begin{theorem}\label{thm:mpq} 
	Let $\varphi(s) = c + r2^{-s}$ with $\mre{c}-1/2\geq r > 0$. Then $\|\mathscr{C}_\varphi\|^2 \leq \zeta(1+\xi)$, where $\xi = (\mre{c}-1/2) + \sqrt{(\mre{c}-1/2)^2-r^2}$. 
\end{theorem}
\begin{proof}
	Let $f(s)=\sum_{n\geq1} a_n n^{-s}$. We combine the Cauchy--Schwarz inequality, for some parameter $\eta>0$ to be chosen later, and Lemma~\ref{lem:dkzeta}, to obtain 
	\begin{multline*}
		|f^{(k)}(c)|^2 \leq \sum_{m=1}^\infty \frac{(\log{m})^k}{m^{\mre{c}+1/2+\eta}} \sum_{n=1}^\infty |a_n|^2 \frac{(\log{n})^k}{n^{\mre{c}-1/2-\eta}} \\
		\leq \frac{k!}{(\mre{c}-1/2+\eta)^k} \zeta(\mre{c}+1/2+\eta) \sum_{n=1}^\infty |a_n|^2 \frac{(\log{n})^k}{n^{\mre{c}-1/2-\eta}}. 
	\end{multline*}
	We insert this estimate into \eqref{eq:ncomp} and note that in the case $\varphi_{\mathbf{c}}(s)=c+r2^{-s}$ we have $\|L_\mathbf{c}\|_{\mathscr{H}^{2k}}^{2k} = r^{2k}$. After changing the order of summation, we find that 
	\begin{multline*}
		\|\mathscr{C}_\varphi f\|_{\mathscr{H}^2}^2 \leq \zeta(\mre{c}+1/2+\eta) \sum_{n=1}^\infty |a_n|^2 \frac{1}{n^{\mre{c}-1/2-\eta}} \sum_{k=0}^\infty \frac{r^{2k}}{k!} \frac{(\log{n})^k}{(\mre{c}-1/2+\eta)^k} \\
		= \zeta(\mre{c}+1/2+\eta) \sum_{n=1}^\infty |a_n|^2 n^{\frac{r^2}{\mre{c}-1/2+\eta} -(\mre{c}-1/2) + \eta}. 
	\end{multline*}
	The proof is completed by letting $\eta = \sqrt{(\mre{c}-1/2)^2-r^2}$. 
\end{proof}

We can combine Theorem~\ref{thm:affinesubord} and Theorem~\ref{thm:mpq} with the Cauchy--Schwarz inequality to obtain the following result. 
\begin{corollary}\label{cor:combo} 
	Suppose that $\varphi(s) = c + \sum_{j\geq1} c_j p_j^{-s}$ is in $\mathscr{G}$. If $r=\sum_{j\geq1}|c_j|$ and $\xi = (\mre{c}-1/2) + \sqrt{(\mre{c}-1/2)^2-r^2}$, then
	\[\|\mathscr{C}_\varphi\|^2 \leq \left(1-\frac{\sum_{j\geq1}|c_j|^2}{r^2}\right)\zeta(2\mre{c}) + \frac{\sum_{j\geq1}|c_j|^2}{r^2} \zeta(1+\xi).\]
\end{corollary}

We end this section by specializing Corollary~\ref{cor:combo} to the symbols $\varphi_{\mathbf{c}}$, $\mathbf{c} \in L(d,r)$, of minimal norm, see \eqref{eq:seqdom} and Theorem~\ref{thm:affinesubord}. 
\begin{corollary}\label{cor:smallnorm} 
	Let $\varphi(s) = c + (r/d)\sum_{j=1}^d p_j^{-s}$ with $\mre{c}-1/2\geq r >0$. Then
	\[\|\mathscr{C}_\varphi\|^2 \leq \zeta(2\mre{c})\left(1+\frac{1}{d}\right).\]
\end{corollary}
\begin{proof}
	From Corollary~\ref{cor:combo} and the fact that $\xi \geq \mre{c}-1/2$, we obtain
	\[\|\mathscr{C}_\varphi\|^2 \leq \zeta(2\mre{c})\left(1+\frac{1}{d}\left(\frac{\zeta(1/2+\mre{c})}{\zeta(2\mre{c})}-1\right)\right).\]
	The upper estimate of Lemma~\ref{lem:dkzeta} with $k=1$ implies that $\sigma \mapsto (\sigma-1)\zeta(\sigma)$ is increasing. Hence
	\[\frac{\zeta(1/2+\mre{c})}{\zeta(2\mre{c})} = 2 \frac{(\mre{c}-1/2)\zeta(1/2+\mre{c})}{(2\mre{c}-1)\zeta(2\mre{c})} \leq 2,\]
	yielding the statement. 
\end{proof}

Lemma~\ref{lem:affinemap} and Corollary~\ref{cor:smallnorm} demonstrate that the norm of a composition operator on $\mathscr{H}^2$ may be made arbitrarily close to the general lower bound \eqref{eq:genlower} without restricting $\varphi(\mathbb{C}_0^\ast)$. We shall see in the next section that the general lower bound $\|\mathscr{C}_\varphi\|^2\geq \zeta(2\mre{c})$ can never be attained unless $\varphi \equiv c$. Note that
\[\varphi(s) = c + \frac{r}{d}\sum_{j=1}^d p_j^{-s}\]
actually converges to $\varphi \equiv c$ in $\mathscr{H}^2$ as $d \to \infty$.

\section{Partial reproducing kernels} \label{sec:rpk} 
The partial reproducing kernel of $\mathscr{H}^2$ generated by $\Lambda \subseteq \mathbb{N}$ is defined by
\[K_w^{\Lambda}(s) = \zeta_{\Lambda}(s+\overline{w}) := \sum_{n\in\Lambda} n^{-s-\overline{w}}.\]
Let $\mathscr{H}^2_{\Lambda}$ denote the corresponding subspace of $\mathscr{H}^2$,
\[\mathscr{H}^2_{\Lambda} := \left\{f \in \mathscr{H}^2 \, : \, f(s) = \sum_{n\in\Lambda} a_n n^{-s}\right\},\]
and let $\sigma(\Lambda)$ denote the abscissa of (absolute) convergence of $\zeta_\Lambda$. If $\Lambda$ is an infinite set, then $0\leq \sigma(\Lambda)\leq 1$. Note that the elements of $\mathscr{H}^2_\Lambda$ are absolutely convergent in $\mathbb{C}_{\sigma(\Lambda)/2}$, by the Cauchy--Schwarz inequality. Moreover, $K^\Lambda_w$ is the reproducing kernel at $w \in \mathbb{C}_{\sigma(\Lambda)/2}$ of the Hilbert space $\mathscr{H}^2_\Lambda$, from which it follows that
\[\|K_w^\Lambda\|_{\mathscr{H}^2_\Lambda}^2 = K_w^\Lambda(w) = \zeta_\Lambda(2\mre{w}).\]
Let $\operatorname{mult}(\Lambda)$ denote the smallest set which contains $\Lambda$ and is closed under multiplication. The following basic lemma is crucial. 
\begin{lemma}\label{lem:partial} 
	Let $\varphi(s) = \sum_{n \in \Lambda'} c_n n^{-s}$ be in $\mathscr{G}$ and set $\Lambda = \operatorname{mult}(\Lambda')$. Then
	\[\|\mathscr{C}_\varphi\|^2 \geq \sup_{\mre{w}>\sigma(\Lambda)/2} \frac{\zeta(2 \mre \varphi(w))}{\zeta_\Lambda(2\mre{w})}.\]
\end{lemma}
\begin{proof}
	Since $\varphi \in \mathscr{G}$ we know that $\mathscr{C}_\varphi f = f \circ \varphi$ is in $\mathscr{H}^2$ for every $f \in \mathscr{H}^2$. Let $\mre{s}>1/2$. By the absolute convergence of $f \circ \varphi$ and the computation
	\[n^{-\varphi(s)} = \sum_{j=0}^\infty \frac{(-\log{n})^j}{j!} (\varphi(s))^j, \qquad n=1,2,\ldots,\]
	we note that $\mathscr{C}_\varphi f$ is in $\mathscr{H}^2_\Lambda$, since $\Lambda = \operatorname{mult}(\Lambda')$. Hence we may consider $\mathscr{C}_\varphi$ as a bounded operator $\mathscr{C}_\varphi \colon \mathscr{H}^2 \to \mathscr{H}^2_\Lambda$ and let $\mathscr{C}_\varphi^\ast\colon \mathscr{H}^2_\Lambda \to \mathscr{H}^2$ denote its adjoint. If $f$ is in $\mathscr{H}^2$ and $\mre{w}>\sigma(\Lambda)/2$ we have that
	\[\langle f, K_{\varphi(w)} \rangle_{\mathscr{H}^2} = f(\varphi(w)) = \langle \mathscr{C}_\varphi f, K_w^\Lambda \rangle_{\mathscr{H}^2_\Lambda} = \langle f, \mathscr{C}_\varphi^\ast K_w^\Lambda \rangle_{\mathscr{H}^2},\]
	and hence $\mathscr{C}_\varphi^\ast K_w^\Lambda = K_{\varphi(w)}$. Using this identity, we obtain the desired estimate,
	\[\|\mathscr{C}_\varphi\|^2 = \|\mathscr{C}_\varphi^\ast\|^2 \geq \sup_{\mre{w}> \sigma(\Lambda)/2} \frac{\|\mathscr{C}_\varphi^\ast K_w^{\Lambda}\|_{\mathscr{H}^2}^2}{\|K_w^\Lambda\|_{\mathscr{H}^2_\Lambda}^2} = \sup_{\mre{w}>\sigma(\Lambda)/2} \frac{\zeta(2\mre{\varphi(w)})}{\zeta_\Lambda(2\mre{w})}. \qedhere\]
\end{proof}

Our first application of Lemma~\ref{lem:partial} is to prove that the general lower bound \eqref{eq:genlower} is not attained unless $\varphi$ is identically constant. This result will be needed in Section~\ref{sec:inner}. 
\begin{theorem}\label{thm:pointeval} 
	Suppose that $\varphi \in \mathscr{G}$ is a non-constant Dirichlet series. Then $\|\mathscr{C}_\varphi \|^2 > \zeta(2\mre{c_1})$, where $c_1 = \varphi(+\infty)$. 
\end{theorem}
\begin{proof}
	Since $\varphi$ is not identically constant, there is an integer $m\geq2$ such that
	\[\varphi(s) = c_1 + \sum_{n=m}^\infty c_n n^{-s}\]
	and $c_m \neq 0$. We may assume that $c_m<0$ by a vertical translation. Let $\Lambda = \{1\}\cup \{n\,:\, n \geq m\}$. By Lemma~\ref{lem:partial} we have that
	\[\|\mathscr{C}_\varphi\|^2 \geq \sup_{1/2 < \sigma < \infty} \frac{\zeta(2\mre{\varphi(\sigma)})}{\zeta_\Lambda(2\sigma)}.\]
	Letting $\sigma \to \infty$ yields the lower bound $\|\mathscr{C}_\varphi\|^2 \geq \zeta(2\mre{c_1})$. Hence it is sufficient to prove that $\sigma \mapsto {\zeta(2\mre{\varphi(\sigma)})}/{\zeta_\Lambda(2\sigma)}$ is eventually decreasing. Logarithmic differentiation leads us to verify that 
	\begin{equation}\label{eq:logdiff} 
		-2 \mre{\varphi'(\sigma)} \frac{\zeta'(2\mre{\varphi(\sigma)})}{\zeta(2\mre{\varphi(\sigma)})} \geq -2 \frac{\zeta_\Lambda'(2\sigma)}{\zeta_\Lambda(2\sigma)} 
	\end{equation}
	holds for all sufficiently large $\sigma$. We now note, since $\zeta'(2\mre{c_1}) < 0$, that
	\[-2 \mre{\varphi'(\sigma)} \frac{\zeta'(2\mre{\varphi(\sigma)})}{\zeta(2\mre{\varphi(\sigma)})} \sim -2 \frac{\zeta'(2\mre{c_1})}{\zeta(2 \mre{c_1})} |c_m| (\log{m}) m^{-\sigma},\]
	as $\sigma \to \infty$. On the other hand $-2 {\zeta_\Lambda'(2\sigma)}/{\zeta_\Lambda(2\sigma)} \sim 4 (\log{m}) m^{-2\sigma}$, establishing \eqref{eq:logdiff} for all sufficiently large $\sigma$. 
\end{proof}

Assume that $\Lambda\subseteq \mathbb{N}$ is such that $\mathscr{C}_\varphi$ maps $\mathscr{H}^2$ into $\mathscr{H}^2_\Lambda$. Set
\[S_\varphi := \sup_{\mre{w}>1/2} \frac{\|\mathscr{C}_\varphi K_w\|_{\mathscr{H}^2}}{\|K_w\|_{\mathscr{H}^2}} \qquad\text{and}\qquad S_\varphi^\ast(\Lambda) := \sup_{\mre{w}>\sigma(\Lambda)/2} \frac{\|\mathscr{C}_\varphi^\ast K_w^{\Lambda}\|_{\mathscr{H}^2}}{\|K_w^\Lambda\|_{\mathscr{H}^2_\Lambda}}.\]
Clearly both quantities constitute lower bounds for $\|\mathscr{C}_\varphi\|$. The proof of the following result is essentially the same as the proof of \cite[Prop.~3.1]{BR01}. 
\begin{lemma}\label{lem:inequalities} 
	Suppose that $\varphi \in \mathscr{G}$ is a non-constant Dirichlet series and that $\Lambda$ is such that $\mathscr{C}_\varphi$ maps $\mathscr{H}^2$ to $\mathscr{H}^2_\Lambda$. Then it holds that $S_\varphi \geq S_\varphi^\ast(\Lambda)$. 
\end{lemma}
\begin{proof}
	Fix $w$ such that $\mre{w}>\sigma(\Lambda)/2$. Since $\Lambda$ is an infinite set, we have that $\sigma(\Lambda)\geq0$. Hence $\varphi(w)$ is in $\mathbb{C}_{1/2}$. We recall that $ \mathscr{C}_\varphi^\ast K_w^\Lambda = K_{\varphi(w)}$ and use the Cauchy--Schwarz inequality to see that 
	\begin{align*}
		\frac{\|\mathscr{C}_\varphi^\ast K_w^\Lambda\|_{\mathscr{H}^2}}{\|K_w^\Lambda\|_{\mathscr{H}^2_\Lambda}} &= \frac{\langle K_{\varphi(w)}, K_{\varphi(w)}\rangle_{\mathscr{H}^2}}{\|K_w^\Lambda\|_{\mathscr{H}^2_\Lambda} \|K_{\varphi(w)}\|_{\mathscr{H}^2}} \\
		&= \frac{\langle \mathscr{C}_\varphi K_{\varphi(w)}, K_w^\Lambda \rangle_{\mathscr{H}^2_\Lambda}}{\|K_w^\Lambda\|_{\mathscr{H}^2_\Lambda} \|K_{\varphi(w)}\|_{\mathscr{H}^2}} \leq \frac{\|\mathscr{C}_\varphi K_{\varphi(w)}\|_{\mathscr{H}^2}}{\|K_{\varphi(w)}\|_{\mathscr{H}^2}}. 
	\end{align*}
	Taking the supremum over $\mre{w}>\sigma(\Lambda)/2$ yields that $S_\varphi^\ast(\Lambda) \leq S_\varphi$. 
\end{proof}
Let us now return to the discussion of the symbol $\varphi(s) = c + r2^{-s}$, where $\mre{c} - 1/2 \geq r > 0$. It was asked in \cite[Sec.~5]{MPQ18} whether 
\begin{equation}\label{eq:wrongconj} 
	\|\mathscr{C}_\varphi\| = S_\varphi = S_\varphi^\ast(\mathbb{N}). 
\end{equation}
However, in this case, $\mathscr{C}_\varphi$ maps $\mathscr{H}^2$ into $\mathscr{H}^2_\Lambda$, for $\Lambda = \{2^j,\,j=0,1,\ldots\}$. Note that
\[\sup_{ \mre{w} > 1/2} \frac{\zeta(2 \mre{ \varphi(w)})}{\zeta(2\mre{w})} = \sup_{ 1/2 < \sigma < \infty} \frac{\zeta(2\mre{c} -2r2^{-\sigma})}{\zeta(2\sigma)},\]
which implies that
\[\lim_{\sigma \to (1/2)^+} \frac{\zeta(2\mre{c} -2r2^{-\sigma})}{\zeta(2\sigma)} = 0 \qquad \text{and} \qquad \lim_{\sigma \to \infty} \frac{\zeta(2\mre{c} -2r2^{-\sigma})}{\zeta(2\sigma)} = \zeta(2\mre{c}).\]
By the proof of Theorem~\ref{thm:pointeval}, we hence find that there is a value $\sigma^\ast \in (1/2, \infty)$ for which
\[(S_\varphi^\ast(\mathbb{N}))^2 = \frac{\zeta(2\mre{c} -2r2^{-\sigma^\ast})}{\zeta(2\sigma^\ast)}.\]
Since $\zeta_{\Lambda}(2\sigma^\ast) < \zeta(2\sigma^\ast)$, we conclude that $S_\varphi^\ast(\Lambda) > S_\varphi^\ast(\mathbb{N})$. Therefore, the second equality in \eqref{eq:wrongconj} could not be true, in view of Lemma~\ref{lem:inequalities}.

More generally, given some symbol $\varphi$, if $\Lambda$ is the minimal set so that $\mathscr{C}_\varphi$ maps $\mathscr{H}^2$ to $\mathscr{H}^2_\Lambda$, we define
\[S_\varphi^\ast := S_\varphi^\ast(\Lambda).\]
Our next goal is to use Lemma~\ref{lem:partial} to prove a new lower bound for $\|\mathscr{C}_\varphi\|$ when $\varphi(s) = c + r2^{-s}$. 
\begin{theorem}\label{thm:adjointlower} 
	Let $\varphi(s) = c + r2^{-s}$ with $\mre{c}-1/2\geq r >0$. Then 
	\begin{equation}\label{eq:adjointlower} 
		(S_\varphi^\ast)^2 = \sup_{0<x<1} (2-x) x\zeta(2\mre{c}-2r(1-x)). 
	\end{equation}
	In particular, 
	\begin{enumerate}
		\item[(a)] it holds that $(S_\varphi^\ast)^2 \geq \xi^{-1}$, where $\xi$ is as in Theorem~\ref{thm:mpq}, and 
		\item[(b)] if $\mre{c}-1/2=r=\xi\leq 1/4$, then $(S_\varphi^\ast)^2 = \xi^{-1}$. 
	\end{enumerate}
\end{theorem}
\begin{proof}
	Clearly $\Lambda = \{2^j \,:\, j=0,1,\ldots\}$ is the smallest possible set such that $\mathscr{C}_\varphi$ maps $\mathscr{H}^2$ into $\mathscr{H}^2_\Lambda$. Moreover, $\sigma(\Lambda)=0$. Hence, by definition and the fact that $\mre{c}$ and $r$ are positive, we see that
	\[(S_\varphi^\ast)^2 = \sup_{\mre{w}>0} \frac{\zeta(2\mre{\varphi(w)})}{\zeta_\Lambda(2\mre{w})} = \sup_{0<\sigma<\infty} (1-2^{-2\sigma}) \zeta(2\mre{c}-2r2^{-\sigma}),\]
	Substituting $x= 1-2^{-\sigma}$ we get \eqref{eq:adjointlower}. To prove (a), we first apply the standard integral estimate $\zeta(\sigma)\geq (\sigma-1)^{-1}$ to obtain
	\[(S_\varphi^\ast)^2 \geq \sup_{0<x<1} \frac{(2-x)x}{2\mre{c}-1-2r(1-x)} \geq \frac{1}{(\mre{c}-1/2) + \sqrt{(\mre{c}-1/2)^2-r^2}} = \frac{1}{\xi},\]
	where we on the basis of a calculus argument chose
	\[x = 1-\frac{(\mre{c}-1/2)+\sqrt{(\mre{c}-1/2)^2-r^2}}{r}.\]
	In the case (b), we have $\mre{c}-1/2=r=\xi$. The lower bound $(S_\varphi^\ast)^2 \geq \xi^{-1}$ is then obtained by letting $x\to0^+$. Hence it is sufficient to prove that
	\[x \mapsto (2-x)x \zeta(1+2\xi x)\]
	is decreasing on $(0,1)$. Logarithmically differentiating the right hand side and multiplying with $-x$ gives the condition
	\[-2\xi x \frac{\zeta'(1+2\xi x)}{\zeta(1+2\xi x)} \geq \frac{4-3x}{2-x}-1, \qquad 0 < x < 1.\]
	To verify this, note that the lower bound in Lemma~\ref{lem:dkzeta} for $k=1$ combined with the estimate $\zeta(\sigma)\geq(\sigma-1)^{-1}$ yields that
	\[-(\sigma-1) \frac{\zeta'(\sigma)}{\zeta(\sigma)} \geq 1-\frac{1}{\zeta(\sigma)} \geq 1-(\sigma-1), \qquad \sigma > 1.\]
	Hence we are done if it holds for all $0 < x < 1$ that
	\[1-2\xi x \geq \frac{4-3x}{2-x} -1 \qquad \Longleftrightarrow \qquad \xi \leq \frac{1}{2(2-x)},\]
	which is clearly true if $0<\xi \leq 1/4$. 
\end{proof}
\begin{remark}
	The restriction $0<\xi \leq 1/4$ in (b) can certainly be improved by more careful estimates, but we cannot have $(S_\varphi^\ast)^2 = \xi^{-1}$ for every $\xi>0$, since we get that $(S_\varphi^\ast)^2 > \zeta(2\mre{c}) = \zeta(1+2\xi)$ from the proof of Theorem~\ref{thm:pointeval}. 
\end{remark}

Let us return to the question asked in \cite[Sec.~5]{MPQ18}, about the validity of \eqref{eq:wrongconj}. Revising the lattermost quantity in \eqref{eq:wrongconj}, we arrive at the following. 
\begin{question}
	If $\varphi(s) = c + r2^{-s}$ is in $\mathscr{G}$ and $r>0$, does it hold that $\|\mathscr{C}_\varphi\| = S_\varphi = S_\varphi^\ast$? 
\end{question}
While we do not have any additional results directly addressing this question, let us briefly discuss what we can say about the analogous question for the family of symbols considered in \cite{Brevig17}. For $0<\alpha<\infty$, let
\[\varphi_\alpha(s) := \frac{1}{2}+\alpha \mathscr{T}(2^{-s}) = \frac{1}{2} + \alpha \frac{1-2^{-s}}{1+2^{-s}},\]
where $\mathscr{T}$ denotes the Cayley transform \eqref{eq:cayley}. The main result in \cite{Brevig17} identifies the symbols $\varphi_\alpha$ as those generating composition operators of maximal norm with $\mre{\varphi(+\infty)} = 1/2 + \alpha$ (see also Theorem~\ref{thm:inner}). In other words, if $\varphi\in\mathscr{G}$ satisfies $\mre{\varphi(+\infty)}=1/2+\alpha$, then for every $f$ in $\mathscr{H}^2$ it holds that 
\begin{equation}\label{eq:alphasub} 
	\|\mathscr{C}_\varphi f\|_{\mathscr{H}^2} \leq \|\mathscr{C}_{\varphi_\alpha} f\|_{\mathscr{H}^2}. 
\end{equation}
Furthermore, it was demonstrated in \cite{Brevig17} that 
\begin{equation}\label{eq:Brevigest} 
	\max\left(\frac{2}{\alpha},\zeta(1+2\alpha)\right) \leq \|\mathscr{C}_{\varphi_\alpha}\|^2 \leq \max\left(\frac{2}{\alpha},\zeta(1+\alpha)\right). 
\end{equation}
The lower bound $2/\alpha$ was obtained by establishing that $(S_{\varphi_\alpha})^2 \geq 2/\alpha$, while the lower bound $\zeta(1+2\alpha)$ is simply the general lower bound \eqref{eq:genlower}. Appealing to the methods of this paper, note that the same considerations as in the proof of Theorem~\ref{thm:adjointlower} show that
\[(S_{\varphi_\alpha}^\ast)^2 = \sup_{0<x<1} \frac{4x}{(1+x)^2} \zeta(1+ 2\alpha x).\]
We thus obtain that $(S_{\varphi_\alpha}^\ast)^2 \geq \max(2/\alpha,\zeta(1+2\alpha))$ by considering $x\to 0^+$ and $x\to 1^-$. This provides a new proof of the lower bound in \eqref{eq:Brevigest} and, for sufficiently large $\alpha$, say, $\alpha \geq 2$, it also yields a small improvement. Combined with the upper bound in \eqref{eq:Brevigest}, this shows that
\[\|\mathscr{C}_{\varphi_\alpha}\| = S_{\varphi_\alpha} = S_{\varphi_\alpha}^\ast = \sqrt{\frac{2}{\alpha}}\]
for all $0< \alpha \leq \alpha_0$, where $\alpha_0 \approx 1.5$ is the unique positive solution to $2 = \alpha\zeta(1+\alpha)$. We do not know whether the first two of these equalities hold also for $\alpha > \alpha_0$.

In analogy with this result, it is tempting to conjecture that $\|\mathscr{C}_\varphi\| = S_\varphi = S_\varphi^\ast$ could hold at least for affine symbols $\varphi(s)= 1/2+\xi(1-2^{-s})$ with sufficiently small $\xi>0$. If this holds, then part (b) of Theorem~\ref{thm:adjointlower} implies that $\|\mathscr{C}_\varphi\|=\xi^{-1/2}$ for all sufficiently small $\xi>0$.

\section{Composition operators on $H^2(\mathbb{T})$ and inner functions} \label{sec:compdisc} 
Let $H^2(\mathbb{T})$ denote the Hardy space of analytic functions $f(z)=\sum_{k\geq0} a_k z^k$ in the unit disc $\mathbb{D} = \mathbb{D}(0,1)$ with square summable coefficients. Every $f \in H^2(\mathbb{T})$ has non-tangential boundary values almost everywhere on $\mathbb{T}$,
\[f(e^{i\theta}) = \lim_{r\to1^-} f(re^{i\theta}).\]
The norm of $H^2(\mathbb{T})$ is given by 
\begin{equation}\label{eq:H2Dnorm} 
	\|f\|_{H^2(\mathbb{T})}^2 := \int_{\mathbb{T}} |f(z)|^2 \, dm(z) = \sum_{k=0}^\infty |a_k|^2. 
\end{equation}
Via non-tangential boundary values, $H^2(\mathbb{T})$ can be viewed as the subspace of $L^2(\mathbb{T})$ of functions whose negative Fourier coefficients all vanish.

Every analytic function $\varphi$ mapping $\mathbb{D}$ into itself generates a composition operator on $H^2(\mathbb{T})$ by $\mathscr{C}_\varphi(f) = f \circ \varphi$. The following well-known norm estimates are sharp, 
\begin{equation}\label{eq:H2Dest} 
	\frac{1}{1-|\varphi(0)|^2} \leq \|\mathscr{C}_\varphi\|^2 \leq \frac{1+|\varphi(0)|}{1-|\varphi(0)|}. 
\end{equation}
The lower bound can be deduced from reproducing kernel arguments, cf. Lemma~\ref{lem:partial}. The upper bound is a consequence of Littlewood's subordination principle, which states that if $\varphi(0)=0$, then 
\begin{equation}\label{eq:littlewood} 
	\|\mathscr{C}_\varphi f\|_{H^2(\mathbb{T})} \leq \|f\|_{H^2(\mathbb{T})}, \qquad f \in H^2(\mathbb{T}). 
\end{equation}
To extend this to the general case $\varphi(0)=w \neq 0$, we use the M\"obius transformation 
\begin{equation}\label{eq:mobius} 
	\psi_w(z) := \frac{w-z}{1-\overline{w}z}, 
\end{equation}
which maps $\mathbb{D}$ onto itself and interchanges the points $0$ and $w$. Writing
\[f \circ \varphi = f \circ \psi_w \circ \psi_w^{-1} \circ \varphi\]
and applying Littlewood's subordination principle to $\psi_w^{-1}\circ \varphi$ yields that 
\begin{equation}\label{eq:littlewoodstrong} 
	\|\mathscr{C}_\varphi f\|_{H^2(\mathbb{T})} \leq \|\mathscr{C}_{\psi_w} f\|_{H^2(\mathbb{T})}, \qquad f \in H^2(\mathbb{T}). 
\end{equation}
The upper bound in \eqref{eq:H2Dest} now follows from the fact that
\[\|\mathscr{C}_{\psi_w}\|^2= \frac{1+|w|}{1-|w|},\]
which can deduced from changing the variables in the integral expression for $\|f \circ \psi_w\|_{H^2(\mathbb{T})}$ from \eqref{eq:H2Dnorm}, together with a simple estimate.

If $\varphi(0)=0$, then the upper and lower bounds coincide and thus $\|\mathscr{C}_\varphi\|=1$. However, this is no longer true in general if we restrict $\mathscr{C}_\varphi$ to $H^2_0(\mathbb{T})$, the subspace of functions $f \in H^2(\mathbb{T})$ with $f(0) = 0$. In this case, J. H.~Shapiro \cite{Shapiro00} has proved the following theorem. Recall that $\varphi$ is said to be inner if $|\varphi(z)| = 1$ for almost every $z \in \mathbb{T}$. 
\begin{theorem}[Shapiro] \label{thm:shapiro1} 
	Suppose that $\varphi$ is an analytic self-map of $\mathbb{D}$ with $\varphi(0) = 0$. Then the following are equivalent. 
	\begin{enumerate}
		\item[(a)] $\varphi$ is inner. 
		\item[(b)] $\mathscr{C}_\varphi \colon H^2(\mathbb{T}) \to H^2(\mathbb{T})$ is an isometry. 
		\item[(c)] $\|\mathscr{C}_\varphi|_{H^2_0(\mathbb{T})}\| = 1$. 
	\end{enumerate}
\end{theorem}
The equivalence of (a) and (b) is due to Nordgren \cite{Nordgren68}. Note in particular that the implication (b) $\implies$ (a) can be deduced by considering the action of $\mathscr{C}_\varphi$ on monomials. Shapiro's insight was to show that $\|\mathscr{C}_\varphi|_{H^2_0(\mathbb{T})}\| < 1$ when $\varphi$ is not inner, by establishing a version of Littlewood's subordination principle \eqref{eq:littlewood} which takes the size of the symbol on $\mathbb{T}$ into account. 

We will now strengthen Shapiro's estimate, by showing that it can be made uniform in the non-innerness of $\varphi$. In preparation, recall that if $\varphi(0) = 0$, a change of variables in the Littlewood--Paley formula for the $H^2(\mathbb{T})$-norm yields that 
\begin{equation}\label{eq:NC} 
	\|\mathscr{C}_\varphi f\|_{H^2(\mathbb{T})}^2 = |f(0)|^2 + 2 \int_{\mathbb{D}} |f'(w)|^2 \, N_\varphi(w)\,dA(w). 
\end{equation}
Here $dA$ is the normalized area measure on $\mathbb{D}$, and $N_\varphi$ is the Nevanlinna counting function, defined by
\[N_\varphi(w) := \sum_{z \in \varphi^{-1}(\{w\})}\log{\frac{1}{|z|}}, \qquad w \neq 0,\]
where preimages are counted with multiplicity. 
\begin{lemma}\label{lem:sharpiro} 
	Suppose that $\varphi \colon \mathbb{D}\to\mathbb{D}$ is analytic and fixes the origin. Let $0\leq \delta\leq 1$, and define $E_\delta := \{z \in \mathbb{T}\,:\, |\varphi(z)| < \delta\}$. For every $f$ in $H^2(\mathbb{T})$ it holds that 
	\begin{equation}\label{eq:sharpiro} 
		\|\mathscr{C}_\varphi f\|_{H^2(\mathbb{T})}^2 \leq C_\delta|f(0)|^2 + (1-C_\delta)\|f\|_{H^2(\mathbb{T})}^2, 
	\end{equation}
	where $C_\delta = \frac{1}{2}\frac{1-\delta}{1+\delta}m(E_\delta)$. 
\end{lemma}
\begin{proof}
	Following the proof of \cite[Thm.~3.2]{Shapiro00}, we begin by defining
	\[\varphi_w(z) := \psi_w \circ \varphi(z) = \frac{w-\varphi(z)}{1-\overline{w}\varphi(z)},\]
	for $w \in \mathbb{D}$ and $z\in\mathbb{T}$, where $\psi_w$ denotes the M\"obius transformation \eqref{eq:mobius}. An elementary computation yields that if $z \in E_\delta$ then
	\[1-|\varphi_w(z)|^2 \geq \frac{1-\delta}{1+\delta}(1-|w|^2).\]
	Combined with the inequality $1-x\leq \log \frac{1}{x}$, valid for $0 < x < 1$, we deduce that 
	\begin{equation}\label{eq:phiwzest} 
		\log{|\varphi_w(z)|} \leq - \frac{1}{2}\frac{1-\delta}{1+\delta}(1-|w|^2), \qquad z \in E_\delta. 
	\end{equation}
	By Jensen's formula and Fatou's lemma we obtain that if $w\neq0$, then
	\[N_{\varphi}(w) \leq \log{\frac{1}{|w|}} + \int_{\mathbb{T}} \log|\varphi_w(z)|\,dm(z) \leq \log{\frac{1}{|w|}} + \int_{E_\delta} \log|\varphi_w(z)|\,dm(z),\]
	since $|\varphi_w(z)|\leq1$ for almost every $z \in \mathbb{T}$. Inserting \eqref{eq:phiwzest} into the latter integral thus yields that 
	\begin{equation}\label{eq:shapiroest} 
		N_\varphi(w) \leq \log{\frac{1}{|w|}}-\frac{1}{2}\frac{1-\delta}{1+\delta} m(E_\delta)(1-|w|^2) = \log \frac{1}{|w|} - C_\delta (1-|w|^2). 
	\end{equation}
	If $f(w)=\sum_{k\geq0} a_k w^k$, then a simple calculation shows that 
	\begin{multline*}
		2\int_{\mathbb{D}} |f'(w)|^2 \left(\log{\frac{1}{|w|}}-C_\delta(1-|w|^2)\right)\,dA(w) \\
		= \sum_{k=1}^\infty |a_k|^2 \left(1-2C_\delta\frac{k}{k+1}\right) \leq (1-C_\delta)\sum_{k=1}^\infty |a_k|^2. 
	\end{multline*}
	Hence, inserting \eqref{eq:shapiroest} into \eqref{eq:NC} yields \eqref{eq:sharpiro}. 
\end{proof}

The following special case of Lemma~\ref{lem:sharpiro} will be used in the next section to improve the upper bound in \eqref{eq:mpqbounds}. 
\begin{lemma}\label{lem:z2z} 
	Let $\psi(z)=z/(2-z)$. For every $f$ in $H^2(\mathbb{T})$ it holds that 
	\begin{equation}\label{eq:z2z} 
		\|\mathscr{C}_\psi f\|_{H^2(\mathbb{T})}^2 \leq \frac{|f(0)|^2+\|f\|_{H^2(\mathbb{T})}^2}{2}. 
	\end{equation}
\end{lemma}

\begin{proof}
	We assume that $f(z) = \sum_{k\geq0}a_k z^k$ and compute
	\[\mathscr{C}_\psi f(z) = \sum_{k=0}^\infty a_k \left(\frac{z}{2-z}\right)^k = a_0 + \sum_{j=1}^\infty \frac{z^j}{2^j} \sum_{k=1}^j \binom{j-1}{k-1} a_k.\]
	Taking the norm and using the Cauchy--Schwarz inequality, we conclude that
	\begin{align*}
		\|\mathscr{C}_\psi f\|_{H^2(\mathbb{T})}^2 &= |a_0|^2 + \sum_{j=1}^\infty \frac{1}{4^j} \left|\sum_{k=1}^j \binom{j-1}{k-1} a_k \right|^2 \\ 
		&\leq |a_0|^2 + \sum_{j=1}^\infty \frac{1}{4^j} \left(\sum_{k=1}^j \binom{j-1}{k-1} |a_k|^2 \right) 2^{j-1} \\
		&=|a_0|^2 + \frac{1}{2} \sum_{k=1}^\infty |a_k|^2 \sum_{j=k}^\infty \binom{j-1}{k-1} \frac{1}{2^j} = |a_0|^2 + \frac{1}{2} \sum_{k=1}^\infty |a_k|^2,
	\end{align*}
	which is the desired estimate \eqref{eq:z2z}, since $a_0=f(0)$.
\end{proof}

\begin{remark}
	The estimate \eqref{eq:z2z} can also be extracted from \cite[Thm.~6.10]{HammondPhD}, but the presented proof is shorter and more direct. In the notation of Lemma~\ref{lem:sharpiro}, the constant $C_\delta = 1/2$ is the best possible for the symbol $\psi(z)=z/(2-z)$. This can be seen by testing $\mathscr{C}_\psi$ on $f(z) = 1/(1-rz)$ and letting $r\to 1^-$.
\end{remark}

On the basis of Theorem~\ref{thm:shapiro1}, using M\"obius transformations in a similar fashion to the derivation of \eqref{eq:littlewoodstrong} from \eqref{eq:littlewood}, Shapiro \cite{Shapiro00} also deduced the following theorem. We have modified the original statement slightly, for easier comparison with our Theorem~\ref{thm:inner}. 
\begin{theorem}[Shapiro] \label{thm:shapiro2} 
	Suppose that $\varphi$ is an analytic self-map of $\mathbb{D}$ with $\varphi(0) = w \neq 0$. Then the following are equivalent. 
	\begin{enumerate}
		\item[(a)] $\varphi$ is inner. 
		\item[(b)] $\|\mathscr{C}_\varphi f\|_{H^2(\mathbb{T})} = \|\mathscr{C}_{\psi_w} f\|_{H^2(\mathbb{T})}$ for every $f \in H^2(\mathbb{T})$. 
		\item[(c)] $\|\mathscr{C}_\varphi\| = \|\mathscr{C}_{\psi_w}\| = \frac{1+|w|}{1-|w|}.$ 
	\end{enumerate}
\end{theorem}

\section{Norm estimates for general symbols} \label{sec:inner} 
The goal of this section is to prove an analogue of Theorem~\ref{thm:shapiro2} for composition operators $\mathscr{C}_\varphi \colon \mathscr{H}^2 \to \mathscr{H}^2$ generated by Dirichlet series symbols $\varphi$ in $\mathscr{G}$. We will rely on Lemma~\ref{lem:sharpiro}, the Riemann mapping theorem, and two different tricks that let us transplant knowledge about composition operators on $H^2(\mathbb{T})$ to the Dirichlet series setting of $\mathscr{H}^2$.

For our first trick, we note that the Dirichlet series $g \in \mathscr{H}^2$ is supported on the integers of the form $n = 2^k$,
\[g(s) = \sum_{k=0}^\infty b_{2^k} 2^{-ks},\]
if and only if there is some $G \in H^2(\mathbb{T})$ such that $g(s)=G(2^{-s})$, and in this case $\|g\|_{\mathscr{H}^2}=\|G\|_{H^2(\mathbb{T})}$. In particular, if $\Phi$ maps $\mathbb{D}$ to $\mathbb{C}_{1/2}$ and $\varphi(s)=\Phi(2^{-s})$, then
\[\|\mathscr{C}_\varphi f\|_{\mathscr{H}^2}=\|\mathscr{C}_\Phi f\|_{H^2(\mathbb{T})}, \qquad f \in \mathscr{H}^2.\]

To demonstrate the virtue of this simple observation, we apply it together with Lemma~\ref{lem:z2z} and \eqref{eq:Brevigest} to obtain an improved upper bound for the composition operator generated by the symbol $\varphi(s)=c+r2^{-s}$, when $\mre{c}-1/2=r=\xi$, cf. Theorem~\ref{thm:mpq}. 
\begin{theorem}\label{thm:newupper} 
	Let $\varphi(s) = c + r2^{-s}$ with $\mre{c}-1/2=r=\xi\geq\alpha_0$, where $\alpha_0$ is the unique positive solution to $2=\alpha\zeta(1+\alpha)$. Then
	\[\|\mathscr{C}_\varphi\|^2 \leq \frac{\zeta(1+2\xi)+\zeta(1+\xi)}{2}.\]
\end{theorem}
\begin{proof}
	Since vertical translations are isometries on $\mathscr{H}^2$, we may replace $\varphi(s)$ by $\varphi(s)+i\tau$ for $\tau\in\mathbb{R}$ without changing $\|\mathscr{C}_\varphi\|$. Hence we may assume that $c$ is real. Let
	\[\Phi(z)= \frac{1}{2}+\xi(1-z) \qquad \text{and} \qquad \Phi_\xi(z)= \frac{1}{2}+\xi\frac{1-z}{1+z}.\]
	These are conformal maps from $\mathbb{D}$ to $\mathbb{D}(1/2+\xi,\xi)$ and $\mathbb{C}_{1/2}$, respectively, satisfying that $\Phi(0)=\Phi_\xi(0)=1/2+\xi$. A computation yields that
	\[\Phi_\xi^{-1} \circ \Phi(z) = \frac{2}{2-z}=\psi(z).\]
	Let $f$ be a Dirichlet polynomial. Note that $f \circ \varphi$ is supported on the integers of the form $n=2^k$. Hence,
	\[\|\mathscr{C}_\varphi f\|_{\mathscr{H}^2}^2 = \|f\circ \Phi\|_{H^2(\mathbb{T})}^2 = \|f \circ \Phi_\xi \circ \psi\|_{H^2(\mathbb{T})}^2 = \|\mathscr{C}_\psi (f \circ \Phi_\xi) \|_{H^2(\mathbb{T})}^2.\]
	Since $f$ is a Dirichlet polynomial, clearly $f \circ \Phi_\xi$ is in $H^2(\mathbb{T})$. Lemma~\ref{lem:z2z} thus gives us that
	\[\|\mathscr{C}_\varphi f\|_{\mathscr{H}^2}^2 \leq \frac{|f\circ \Phi_\xi(0)|^2+\|\mathscr{C}_{\Phi_\xi} f\|_{H^2(\mathbb{T})}^2}{2}= \frac{|f(1/2+\xi)|^2+\|\mathscr{C}_{\varphi_\xi} f\|_{\mathscr{H}^2}^2}{2},\]
	where, as before, $\varphi_{\xi}(s) = \Phi_\xi(2^{-s})$. The proof is completed by applying the Cauchy--Schwarz inequality, the upper bound in \eqref{eq:Brevigest}, and the density of Dirichlet polynomials in $\mathscr{H}^2$. 
\end{proof}

Our second trick takes its starting point in Lemma~\ref{lem:carlesoncomp}. To make use of it, we introduce a dummy variable $w \in \mathbb{T}$, acting coordinate-wise on $\chi \in \mathbb{T}^\infty$, allowing us to apply Lemma~\ref{lem:sharpiro} to the Hardy space of functions in the dummy variable. We begin with the following lemma, closely related to \cite[Thm. 4.1]{HLS97}. 
\begin{lemma}\label{lem:wtrick} 
	Suppose that $f \in H^2(\mathbb{T}^\infty)$. For $w \in \mathbb{T}$ and $\chi \in \mathbb{T}^\infty$, let
	\[\chi_w = (w \chi_1,w\chi_2,\ldots).\]
	Then, for almost every $\chi\in\mathbb{T}^\infty$, the function $F_\chi(w) := f(\chi_w)$ is in $H^2(\mathbb{T})$ and
	\[F_\chi(0) = \int_{\mathbb{T}^\infty} f(\chi')\,dm_\infty(\chi').\]
\end{lemma}
\begin{proof}
	The fact that $m_\infty$ is invariant under rotations and Fubini's theorem yields that
	\[\|f\|_{H^2(\mathbb{T}^\infty)}^2 = \int_{\mathbb{T}} \int_{\mathbb{T}^\infty} |f(\chi_w)|^2\,dm_\infty(\chi) \, dm(w) = \int_{\mathbb{T}^\infty}\int_{\mathbb{T}} |F_\chi(w)|^2\,dm(w) \, dm_\infty(\chi). \]
	We conclude that $F_\chi \in L^2(\mathbb{T})$ for almost every $\chi$. To additionally see that $F_\chi$ is in $H^2(\mathbb{T})$, we need to verify that $\widehat{F_\chi}(k)=0$ for almost every $\chi$ and every $k < 0 $, where
	\[\widehat{F_\chi}(k) = \int_{\mathbb{T}} F_\chi(w) \overline{w^{k}}\,dm(w).\]
	To show this, note that for every $k<0$ and $n \in \mathbb{N}$ we have that 
	\begin{align*}
		\int_{\mathbb{T}^\infty} \widehat{F_\chi}(k) \overline{\chi(n)}\,dm_\infty(\chi) &= \int_{\mathbb{T}} \left( \int_{\mathbb{T}^\infty} f(\chi_w) \overline{\chi(n)} \, dm_\infty(\chi)\right) \overline{w^k} \, dm(w) \\
		&= \int_{\mathbb{T}} \left( \int_{\mathbb{T}^\infty} f(\chi) \overline{\chi(n)} \, dm_\infty(\chi)\right) \overline{w^{k-\kappa(n)}} \, dm(w) = 0, 
	\end{align*}
	where $\kappa(n) \geq 0$ denotes the number of prime factors of $n$, counting multiplicities. Similarly,
	\[\int_{\mathbb{T}^\infty} \widehat{F_\chi}(k) \overline{\chi(q)}\,dm_\infty(\chi) = 0, \qquad q \in \mathbb{Q}_+ \backslash \mathbb{N},\]
	since the corresponding Fourier coefficient of $f$ is zero. Hence, the function $\chi \mapsto \widehat{F_\chi}(k)$ is zero for almost every $\chi \in \mathbb{T}^\infty$, since all of its Fourier coefficients vanish.
	
	With $k = 0$, the same calculation also shows that
	\[F_\chi(0) = \widehat{F_\chi}(0) = \int_{\mathbb{T}} F_\chi(w)\,dm(w) = \int_{\mathbb{T}^\infty} f(\chi')\,dm_\infty(\chi'),\]
	for almost every $\chi \in \mathbb{T}^\infty$. 
\end{proof}

Our next goal is to extend Lemma~\ref{lem:sharpiro} to composition operators on $\mathscr{H}^2$, with a statement adapted to a domain $\Omega \supseteq \varphi(\mathbb{C}_0^\ast)$. For simplicity, considering $\Omega\subseteq\mathbb{C}_{1/2}$ as an open set on the Riemann sphere $\mathbb{C}^\ast$, we assume that $\partial \Omega$ is a Jordan curve on $\mathbb{C}^\ast$. Then any Riemann map $\Theta$ of $\mathbb{D}$ onto $\Omega$ extends to a homeomorphism of $\overline{\mathbb{D}}$ onto $\overline{\Omega}$. 
\begin{lemma}\label{lem:H2sharpiro} 
	Suppose that $\Omega \subseteq \mathbb{C}_{1/2}$ has Jordan curve boundary in $\mathbb{C}^\ast$ and that $\varphi\in\mathscr{G}$ maps $\mathbb{C}_0^\ast$ into $\Omega$ with $\varphi(+\infty)=\omega$. Let $\Theta$ be a Riemann map from $\mathbb{D}$ to $\Omega$ with $\Theta(0)=\omega$, and for $0\leq \delta \leq 1$, set
	\[\mathscr{E}_\delta := \left\{\chi \in \mathbb{T}^\infty\,:\, |\Theta^{-1}\circ \varphi^\ast(\chi)| < \delta\right\}.\]
	Then 
	\begin{equation}\label{eq:H2sharpiro} 
		\|\mathscr{C}_\varphi f\|_{\mathscr{H}^2}^2 \leq C_\delta |f(\omega)|^2 + (1-C_\delta)\|\mathscr{C}_\psi f\|_{\mathscr{H}^2}^2, \qquad f \in \mathscr{H}^2, 
	\end{equation}
	where $C_\delta = \frac{1}{2}\frac{1-\delta}{1+\delta}m_\infty(\mathscr{E}_\delta)$ and $\psi(s)=\Theta(2^{-s})$. 
\end{lemma}
\begin{proof}
	Let $f$ be a Dirichlet polynomial. As in the proof Theorem~\ref{thm:newupper}, we will rely on the fact that $\|f\circ \psi\|_{\mathscr{H}^2} = \|f \circ \Theta\|_{H^2(\mathbb{T})}$. We use Lemma~\ref{lem:carlesoncomp}, and that $\Theta$ is a Riemann map, to see that 
	\begin{equation}\label{eq:Thetain} 
		\|\mathscr{C}_\varphi f\|_{\mathscr{H}^2}^2 = \int_{\mathbb{T}^\infty} |f\circ \varphi^\ast(\chi)|^2\,dm_\infty(\chi) = \int_{\mathbb{T}^\infty} |f \circ \Theta \circ \Theta^{-1} \circ \varphi^\ast(\chi)|^2\,dm_\infty(\chi). 
	\end{equation}
	For $w \in \mathbb{T}$, set $\chi_w=(w\chi_1,w\chi_2,\ldots)$. Since $\Theta$ extends to a homeomorphism of $\overline{\mathbb{D}}$ onto $\overline{\Omega}$, Lemma~\ref{lem:wtrick} and the maximum principle implies that the function
	\[\Phi_\chi(w) := \Theta^{-1} \circ \varphi^\ast(\chi_w)\]
	is an analytic self-map of $\mathbb{D}$ with $\Phi_\chi(0)=0$, for almost every $\chi \in \mathbb{T}^\infty$. The measure $m_\infty$ is rotationally invariant, so from \eqref{eq:Thetain} and Fubini's theorem we find that 
	\begin{equation}\label{eq:dolores} 
		\|\mathscr{C}_\varphi f\|_{\mathscr{H}^2} = \int_{\mathbb{T}^\infty}\int_{\mathbb{T}}|f\circ \Theta \circ \Phi_\chi(w)|^2\,dm(w) \, dm_\infty(\chi). 
	\end{equation}
	Let $\chi$ belong to the set of full measure such that $\Phi_\chi$ extends to an analytic self-map of $\mathbb{D}$. Lemma~\ref{lem:sharpiro} yields that 
	\begin{multline*}
		\int_{\mathbb{T}}|f\circ \Theta \circ \Phi_\chi(w)|^2\,dm(w) \leq C_\delta(\chi) |f\circ\Theta(0)|^2 + (1-C_\delta(\chi))\|f \circ \Theta\|_{H^2(\mathbb{T})}^2 \\
		= C_\delta(\chi) |f(\omega)|^2 + (1-C_\delta(\chi))\|\mathscr{C}_\psi f\|_{\mathscr{H}^2}^2, 
	\end{multline*}
	where $C_\delta(\chi) = \frac{1}{2}\frac{1-\delta}{1+\delta} m(E_\delta(\chi))$, and $E_\delta(\chi)=\{w \in \mathbb{T} \,:\, |\Phi_\chi(w)| < \delta\}$. Note that
	\[\int_{\mathbb{T}^\infty} m(E_\delta(\chi))\,dm_\infty(\chi) = m\times m_\infty \left(\left\{(w,\chi)\in\mathbb{T}\times\mathbb{T}^\infty\,:\, |\Phi_\chi(w)| < \delta \right\}\right)= m_\infty(\mathscr{E}_\delta),\]
	by Fubini's theorem and the rotational invariance of $m_\infty$. Inserting the last estimate into \eqref{eq:dolores} thus yields \eqref{eq:H2sharpiro}, at first for Dirichlet polynomials $f$, and by density for all $f \in \mathscr{H}^2$. 
\end{proof}
When $\Omega = \mathbb{D}(c,r)$ and $\omega =c$, the function $\psi$ in the statement is of course given by $\psi(s) = c + r2^{-s}$. Hence \eqref{eq:H2sharpiro} extends the estimate \eqref{eq:effectivesubord} of Theorem~\ref{thm:affinesubord} to non-affine maps, with some loss of precision.

We now present the main result of this section, which identifies the symbols with prescribed mapping properties that are maximal with respect to subordination of composition operators. We say that a Dirichlet series $f \in \mathscr{H}^2$ is inner if $|f^\ast(\chi)| = 1$ for almost every $\chi \in \mathbb{T}^\infty$. 
\begin{theorem}\label{thm:inner} 
	Suppose that $\Omega \subseteq \mathbb{C}_{1/2}$ has Jordan curve boundary in $\mathbb{C}^\ast$ and that $\varphi\in\mathscr{G}$ maps $\mathbb{C}_0^\ast$ into $\Omega$ with $\varphi(+\infty)=\omega$. Let $\Theta$ be a Riemann map from $\mathbb{D}$ to $\Omega$ with $\Theta(0)=\omega$ and set $\psi(s)=\Theta(2^{-s})$. Then
	\[\|\mathscr{C}_\varphi f\|\leq \|\mathscr{C}_\psi f\|, \qquad f \in \mathscr{H}^2.\]
	Furthermore, the following are equivalent. 
	\begin{enumerate}
		\item[(a)] $\Theta^{-1} \circ \varphi$ is inner. 
		\item[(b)] $\|\mathscr{C}_\varphi f\|_{\mathscr{H}^2} = \|\mathscr{C}_{\psi} f\|_{\mathscr{H}^2}$ for every $f \in \mathscr{H}^2$. 
		\item[(c)] $\|\mathscr{C}_\varphi\| = \|\mathscr{C}_{\psi}\|$. 
	\end{enumerate}
\end{theorem}
\begin{proof}
	For $\chi \in \mathbb{T}^\infty$, define, as in the proof of Lemma~\ref{lem:H2sharpiro},
	\[\Phi_\chi(w) = \Theta^{-1} \circ \varphi^\ast(\chi_w), \qquad w \in \mathbb{T}.\]
	Since $\Phi_\chi \colon \mathbb{D} \to \mathbb{D}$ satisfies $\Phi_\chi(0) = 0$, it follows from \eqref{eq:dolores} and Littlewood's subordination principle \eqref{eq:littlewood} that
	\[\|\mathscr{C}_\varphi f\|_{\mathscr{H}^2} \leq \int_{\mathbb{T}^\infty}\int_{\mathbb{T}}|f\circ \Theta(w)|^2\,dm(w) \, dm_\infty(\chi) = \|f \circ \Theta\|^2_{H^2(\mathbb{T})} = \|\mathscr{C}_\psi f\|^2_{\mathscr{H}^2},\]
	for every Dirichlet polynomial $f$. If $\Theta^{-1} \circ \varphi$ is inner, then $\Phi_\chi \in H^2(\mathbb{T})$ is inner for almost every $\chi \in \mathbb{T}^\infty$, again with $\Phi_\chi(0)=0$. Hence by the part of Theorem~\ref{thm:shapiro1} due to Nordgren \cite{Nordgren68}, we know that $\mathscr{C}_{\Phi_\chi}$ is an isometry on $H^2(\mathbb{T})$. Inserting this into \eqref{eq:dolores} thus yields an equality in this case,
	\[\|\mathscr{C}_\varphi f\|_{\mathscr{H}^2}^2 = \|f\circ \Theta\|_{H^2(\mathbb{T})}^2 = \|\mathscr{C}_\psi f\|_{\mathscr{H}^2}^2.\]
	It remains to prove that (c) $\implies$ (a). Suppose that $\Theta^{-1} \circ \varphi$ is not inner, and let $f$ be a Dirichlet polynomial. Then it follows from Lemma~\ref{lem:H2sharpiro} and the Cauchy-Schwarz inequality that there is a constant $0<C<1$ such that
	\[\|\mathscr{C}_\varphi f\|_{\mathscr{H}^2}^2 \leq C|f(\omega)|^2 + (1-C) \|\mathscr{C}_\psi f\|_{\mathscr{H}^2}^2 \leq \|f\|_{\mathscr{H}^2}^2 \left(C\zeta(2\mre{\omega})+(1-C)\|\mathscr{C}_\psi\|^2\right).\]
	Combining this with Theorem~\ref{thm:pointeval} applied to $\mathscr{C}_\psi$, we obtain the strict inequality
	\[\|\mathscr{C}_\varphi\|^2 \leq C\zeta(2\mre{\omega})+(1-C)\|\mathscr{C}_\psi\|^2 < (1-C)\|\mathscr{C}_\psi\|^2+C\|\mathscr{C}_\psi\|^2 = \|\mathscr{C}_\psi\|^2. \qedhere\]
\end{proof}

In the case that $\Omega = \mathbb{C}_{1/2}$ and $\mre{\omega}=1/2+\alpha$, Theorem~\ref{thm:inner} sharpens \eqref{eq:alphasub}. The inequality \eqref{eq:alphasub} was originally used in \cite{Brevig17} to address Problem 3 in \cite{Hedenmalm04}. This problem asks about the maximal norm of a composition operator $\mathscr{C}_\varphi$, given that $\mre{\varphi(+\infty)} = 1/2 + \alpha$.

Theorem~\ref{thm:inner} is the complete analogue of Shapiro's Theorem~\ref{thm:shapiro2}. A corresponding analogue of Shapiro's Theorem~\ref{thm:shapiro1} would concern composition operators $\mathscr{C}_\varphi$ for symbols $\varphi \in \mathscr{G}$ with $c_0 \geq 1$, where
\[\varphi(s) = c_0 s + \sum_{n=1}^\infty c_n n^{-s} = c_0s + \varphi_0(s).\]
However, there is no simple way to transfer results between the cases $c_0 = 0$ and $c_0 \geq 1$. This is unlike the classical setting, where M\"obius transformations allow us to relate maps $\Phi \colon \mathbb{D} \to \mathbb{D}$ satisfying $\Phi(0) = w \neq 0$ with those satisfying $\Phi(0) = 0$. When $c_0 \geq 1$, Bayart \cite[Thm.~16]{Bayart02} has proven the analogue of Nordgren's theorem: $\mathscr{C}_\varphi \colon \mathscr{H}^2 \to \mathscr{H}^2$ is an isometry if and only if $\mathscr{T}^{-1} \circ \varphi_0$ is inner, where $\mathscr{T}$ denotes the Cayley transform \eqref{eq:cayley}.

\section{Examples} \label{sec:examples} 
In this final section we give three examples. Fix $\mre{c}-1/2\geq r > 0$ and set $\Omega = \mathbb{D}(c,r)$. The composition operators we consider will be generated by symbols $\varphi$ mapping $\mathbb{C}_0^\ast$ to $\Omega$ with $\varphi(+\infty)=c$. Accordingly, we let $\Theta(z)=c+rz$.

\subsection{} The first symbol we consider is
\[\varphi(s)=c+\frac{r}{2\sqrt{2}}\left(2^{-s}+2i\, 4^{-s}+8^{-s}\right),\]
which contains (powers of) only one prime number. Therefore the boundary value $\varphi^\ast$ is a function of only one variable $\chi_1 = e^{i\theta_1}$. On the basis of the calculation
\[\left|\Theta^{-1} \circ \varphi^\ast(\chi_1)\right|=\frac{1}{2\sqrt{2}}\left|e^{i\theta_1}+2i\, e^{2i\theta_1}+e^{3i\theta_1}\right|=\sqrt{\frac{\cos^2(\theta_1)+1}{2}}\]
we find that $\varphi$ maps $\mathbb{C}_0^\ast$ into $\mathbb{D}(c,r)$. Theorem~\ref{thm:inner} and Theorem~\ref{thm:mpq} thus yield that $\|\mathscr{C}_\varphi\|^2 < \zeta(1+\zeta)$, where, as always, $\xi = (\mre{c}-1/2)+ \sqrt{(\mre{c}-1/2)^2-r^2}$. 

We can use Lemma~\ref{lem:H2sharpiro} to give a better estimate. The calculation above also yields that if $1/\sqrt{2}\leq \delta \leq 1$, then
\[C_\delta = \frac{1}{2}\frac{1-\delta}{1+\delta} m_\infty(\mathscr{E}_\delta) = \frac{1}{2}\frac{1-\delta}{1+\delta}\left(1-\frac{2}{\pi}\arccos\left(\sqrt{2\delta^2-1}\,\right)\right).\]
A reasonable choice is $\delta=\sqrt{5/8}$. Hence we obtain that
\[\|\mathscr{C}_\varphi\|^2 \leq C \zeta(2\mre{c})+(1-C)\zeta(1+\xi), \qquad C = \frac{13-4\sqrt{10}}{18}=0.01949\ldots\]

\begin{figure}
	\begin{tikzpicture}
		[scale=0.725] 
		\begin{axis}
			[axis equal image, hide axis, ymin=-1.1, ymax=1.1, xmin=-1.1, xmax=1.1,]
			
			\addplot[domain=-200:200, samples = 4000, color=gray, thin] ({cos(((180)/pi)*ln(2)*x)*(3/4)+cos(((180)/pi)*ln(3)*x)*(1/4)},{sin(((180)/pi)*ln(2)*x)*(3/4)+sin(((180)/pi)*ln(3)*x)*(1/4)}); 
		\end{axis}
	\end{tikzpicture}
	\hspace{0.4cm} 
	\begin{tikzpicture}
		[scale=0.725] 
		\begin{axis}
			[axis equal image, hide axis, ymin=-1.1, ymax=1.1, xmin=-1.1, xmax=1.1,]
			
			\addplot[domain=-200:200, samples = 4000, color=gray, thin] ({cos(((180)/pi)*ln(2)*x)*(1/2)+cos(((180)/pi)*ln(3)*x)*(1/2)},{sin(((180)/pi)*ln(2)*x)*(1/2)+sin(((180)/pi)*ln(3)*x)*(1/2)}); 
		\end{axis}
	\end{tikzpicture}
	\hspace{0.4cm} 
	\begin{tikzpicture}
		[scale=0.725] 
		\begin{axis}
			[axis equal image, hide axis, ymin=-1.1, ymax=1.1, xmin=-1.1, xmax=1.1,]
			
			\addplot[domain=-200:200, samples = 4000, color=gray, thin] ({cos(((180)/pi)*ln(2)*x)*(2/3)+cos(((180)/pi)*ln(3)*x)*(1/6)+cos(((180)/pi)*ln(5)*x)*(1/6)},{sin(((180)/pi)*ln(2)*x)*(2/3)+sin(((180)/pi)*ln(3)*x)*(1/6)+sin(((180)/pi)*ln(5)*x)*(1/6)});
			
			\addplot[domain=0:360, samples = 100, color=gray, thin, densely dotted] ({cos(x)*(1/3)},{sin(x)*(1/3)}); 
		\end{axis}
	\end{tikzpicture}
	\caption{Plots of $\varphi_{\mathbf{c}}(it)$ for $-200 \leq t \leq 200$ with $\mathbf{c}$ equal to $\frac{r}{4}(3,1,0)$, $\frac{r}{2}(1,1,0)$ and $\frac{r}{6}(4,1,1)$. The inner radius is $\frac{r}{2}$, $0$ and $\frac{r}{3}$.} 
	\label{fig:annulus} 
\end{figure}
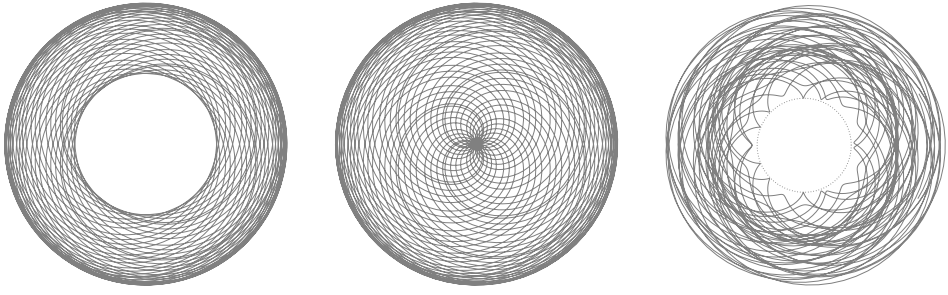

\subsection{} If there are $d\geq2$ primes in the symbol $\varphi$, it is generally much more difficult to estimate $m_\infty(\mathscr{E}_\delta)$. However, using the ergodic theorem, it is possible to express $m_\infty(\mathscr{E}_\delta)$ by looking only at the imaginary axis. By for example \cite[Thm.~2.1.12]{QQ13} we find that the equality
\[m_\infty(\mathscr{E}_\delta) = \lim_{T\to\infty} \frac{1}{2T}\operatorname{meas}\left(\left\{t \in [-T,T]\,:\, \left|\Theta^{-1}\circ \varphi_\chi(it)\right| < \delta \right\}\right)\]
is valid for almost every $\chi \in \mathbb{T}^\infty$. If $\Theta^{-1}\circ \varphi$ is a Dirichlet polynomial, then the equality actually holds for every $\chi \in \mathbb{T}^\infty$, in particular for $\chi\equiv1$. 

As an example, we will consider the following affine symbols: 
\begin{align*}
	\varphi_1(s) &= c + \tfrac{r}{4}(3\cdot 2^{-s}+3^{-s}), \\
	\varphi_2(s) &= c + \tfrac{r}{2}(2^{-s}+3^{-s}), \\
	\varphi_3(s) &= c + \tfrac{r}{6}(4\cdot 2^{-s}+3^{-s}+5^{-s}). 
\end{align*}
Recall from \cite[Sec.~XI.5]{Titchmarsh86} that if $\varphi(s)=c + \sum_{j\geq1} c_j p_j^{-s}$ with $c_1 \geq c_2 \geq \cdots \geq 0$, then the closure of the image of the imaginary axis is an annulus,
\[\overline{\varphi(i\mathbb{R})} = \left\{w \in \mathbb{C}\,:\, r_0 \leq |w-c| \leq r\right\},\qquad r = \sum_{j=1}^\infty c_j, \qquad r_0=\max(0,2c_1-r).\]

In Figure~\ref{fig:annulus} we plot $\varphi_j(it)$ for $j=1,2,3$ and $-200\leq t \leq 200$. Theorem~5 reveals that $\|\mathscr{C}_{\varphi_1} f \|_{\mathscr{H}^2} \geq \|\mathscr{C}_{\varphi_j} f\|_{\mathscr{H}^2}$ for $j=2,3$ and every $f \in \mathscr{H}^2$, but it does not yield any conclusion regarding the relationship between the composition operators generated by $\varphi_2$ and $\varphi_3$. Inspection of their plots near the outer radius might lead to the conjecture that $\mathscr{C}_{\varphi_3}$ is subordinate to $\mathscr{C}_{\varphi_2}$.

To verify this conjecture, we recall from the proof of Theorem~\ref{thm:affinesubord} that it is sufficient to prove that $\|\varphi_3 -c\|_{\mathscr{H}^{2k}}^{2k} \leq \|\varphi_2 -c \|_{\mathscr{H}^{2k}}^{2k}$ holds for $k=1,2,\ldots$. By the multinomial theorem and a computation, this set of inequalities is equivalent to the statement that
\[\sum_{j_1+j_2+j_3=k} \binom{k}{j_1,j_2,j_3}^2 16^{j_1} \leq 9^{k} \binom{2k}{k}, \qquad k = 1,2 \ldots.\]
This inequality can be checked by hand for small $k$. For large $k$, it can be proven by using Stirling's formula and the Laplace method, see for example \cite{RS09} for a calculation of the asymptotic behavior of a very similar sum. Note that the largest term of the sum occurs when $j_1 \approx 2k/3$ and $j_2 \approx j_3 \approx k/6$. We omit the details.

Hence we find that $\mathscr{C}_{\varphi_3}$ is indeed subordinate to $\mathscr{C}_{\varphi_2}$. It would be interesting to know if a more general subordination principle than Theorem~\ref{thm:affinesubord} holds, still operating in the family of affine composition operators with the same mapping properties. In particular, we pose the following question.

\begin{question}
	If $\varphi_{\mathbf{b}}$ and $\varphi_{\mathbf{c}}$ are two affine symbols with the same mapping properties, is it true that either $\mathscr{C}_{\varphi_{\mathbf{b}}}$ is subordinate to $\mathscr{C}_{\varphi_{\mathbf{c}}}$, or $\mathscr{C}_{\varphi_{\mathbf{c}}}$ is subordinate to $\mathscr{C}_{\varphi_{\mathbf{b}}}$?
\end{question}
By Theorem~\ref{thm:affinesubord} a counter-example would have to contain $d\geq3$ prime numbers.

\subsection{} The inner functions on $\mathbb{T}^2$ are already difficult to describe or classify, see for example \cite{BK13}. To give an example of a non-trivial inner function $g \in \mathscr{H}^2$, consider
\[g(s) = \exp \left( -\sum_{j=1}^\infty \lambda_j \frac{e^{i\theta_j} + p_j^{-s}}{e^{i\theta_j} - p_j^{-s}}\right),\]
where $(e^{i\theta_1}, e^{i\theta_2}, \ldots) \in \mathbb{T}^\infty$ and the sequence $(\lambda_j)_{j\geq1}$ is non-negative and summable. This Dirichlet series $g$ is inner, does not have any zeroes in $\mathbb{C}_0$, and fails to converge on the imaginary axis. The symbol
\[\varphi(s)= c + r\frac{g(s)-g(+\infty)}{1-\overline{g(+\infty)}g(s)},\]
is in $\mathscr{G}$ and maps $\mathbb{C}_0^\ast$ to $\mathbb{D}(c,r)$ with $\varphi(+\infty)=c$. If $\psi(s)=\Theta(2^{-s})=c+r2^{-s}$, then Theorem~\ref{thm:inner} yields that $\|\mathscr{C}_\varphi f\|_{\mathscr{H}^2}=\|\mathscr{C}_\psi f\|_{\mathscr{H}^2}$ for every $f \in \mathscr{H}^2$. 

\bibliographystyle{amsplain} 
\bibliography{circles} 
\end{document}